\documentclass[8pt,reqno]{amsart}

\usepackage{amsmath, amsfonts, amssymb, latexsym, amsthm}
\usepackage[pagewise]{lineno}
\usepackage{amsmath, amsfonts, amssymb, latexsym}
\usepackage[numbers,sort&compress]{natbib}
\usepackage{mathrsfs}
\usepackage{pdfcomment}
\hypersetup{hidelinks,
	colorlinks=true,
	allcolors=black,
	pdfstartview=Fit,
	breaklinks=true
}

\everymath{\displaystyle}

\newtheorem{theorem}{Theorem}
\theoremstyle{plain}

\newtheorem{lemma}[theorem]{Lemma}

\newtheorem{proposition}[theorem]{Proposition}

\newtheorem{remark}[theorem]{Remark}

\numberwithin{equation}{section}
\numberwithin{theorem}{section}

\newcommand{\cE}{\mathcal{E}}

\newcommand{\cL}{\mathcal{L}}

\newcommand{\R}{\mathbb{R}}

\newcommand{\N}{\mathbb{N}}

\def\e{\varepsilon}

\begin{document}
\title
{Null controllability of two kinds of coupled parabolic systems with switching control}

\author{\sffamily Yuanhang Liu$^{1}$, Weijia Wu$^{1,*}$, Donghui Yang$^1$   \\
	{\sffamily\small $^1$ School of Mathematics and Statistics, Central South University, Changsha 410083, China. }\\
}
	\footnotetext[1]{Corresponding author: weijiawu@yeah.net }

\email{liuyuanhang97@163.com}
\email{weijiawu@yeah.net}
\email{donghyang@outlook.com}


\maketitle

\begin{abstract}
The focus of this paper is on the null controllability of two kinds of coupled systems including both degenerate and non-degenerate equations with switching control.  We first establish the observability inequality for measurable subsets in time for such coupled system, and then by the HUM method to obtain the null controllability. Next, we investigate the null controllability of such coupled system for segmented time intervals. Notably, these results are obtained through spectral inequalities rather than using the method of Carleman estimates. Such coupled systems with switching control, to the best of our knowledge, are among the first to discuss.
\end{abstract}

\pagestyle{myheadings}
\thispagestyle{plain}
\markboth{CONTROLLABILITY OF COUPLED PARABOLIC SYSTEMS}{YUANHANG LIU, WEIJIA WU AND DONGHUI YANG}

\section{Introduction}
Controlling coupled parabolic systems presents a challenge that has captured the attention of the control community for several decades. These parabolic systems, including both degenerate and non-degenerate ones, have been encountered in the study of chemical reactions (see, e.g., \cite{bothe2003reaction,erdi1989mathematical}), as well as in various physical and mathematical biology scenarios (see, e.g., \cite{hillen2009user,lauffenburger1982effects,shigesada1979spatial}). On one hand, in \cite{gonzalez2010controllability}, the controllability properties of a linear coupled non-degenerate parabolic system comprising $m$ equations were analyzed under a unique distributed control. The null controllability property for the system with a single control force was established by proving a global Carleman inequality for the adjoint system. For other controllability issues related to non-degenerate coupled parabolic equations, see \cite{ammar2011recent} and references therein. On the other hand, in \cite{cannarsa2009controllability}, the null controllability properties were studied for two systems of coupled one-dimensional degenerate parabolic equations for the first time. The first system consists of two forward equations, while the second consists of one forward equation and one backward equation. Both systems are in cascade, meaning the solution of the first equation acts as a control for the second equation, and the control function only directly influences the first equation. In \cite{hajjaj2013carleman}, the null controllability of weakly degenerate coupled parabolic systems with two different diffusion coefficients and one control force was investigated by utilizing global Carleman estimates. In \cite{wu2020null}, the null controllability of a system of $m$ linear weakly degenerate parabolic equations with coupling terms of first and zero order and only one control force was addressed. Due to the degeneracy, the study was transferred to an approximate non-degenerate adjoint system, and a uniform Carleman estimate and an observation inequality for this approximate adjoint system were obtained. For other controllability issues related to coupled parabolic equations, see \cite{liu2014controllability,liu2018controllability} and references therein.

It should be noted that control systems in practical applications often have multiple actuators. Hence, there is a need to develop switching control strategies that ensure only one actuator is activated at any given time. The study of switching controllers has been extensively pursued in diverse fields of application (refer to the survey article \cite{shorten2007stability}). The notion of switching may also pertain to the capacity of the state equation to transit from one configuration to another at certain time instances (refer to \cite{hante2009modeling} for an application to transportation networks). Further controllability issues pertaining to switching control can be found in \cite{lu2014robust,zuazua2010switching} and references therein.

Thus far, there has been no research on the controllability of coupled systems with switching control involving both degenerate and non-degenerate equations. This paper is the first to attempt to address such coupled systems with switching control. Compared to coupled systems involving either degenerate or non-degenerate equations, studying the controllability of coupled systems with switching control involving both types of equations presents a greater challenge. This is because the selection of the weight function using Carleman estimates can be difficult. Therefore, we employ a spectral inequality method (see e.g., \cite{wang2008null,apraiz2014observability}) to solve this problem. It is noteworthy that we derive null controllability results not only for the scenario of weak degeneracy, but also for the instance of strong degeneracy.

Before we state our main theorems, let us introduce necessary notations.

Let
$T>0$ be a fixed positive time constant,
and $I:=(0,1)$. $G_1,G_2$ are nonempty and open subsets of $I$ such that $G_1\cap G_2=\emptyset$. Write $\chi_{G_1},\chi_{G_2}$ for the characteristic functions of $G_1,G_2$, respectively.

Throughout this paper, we denote by
$\langle\cdot, \cdot\rangle$ the inner product in $L^2(I)$, and denote by
$\|\cdot\|$ the norm induced by $\langle\cdot,\cdot\rangle$. We denote by $|\cdot|$
the Lebesgue measure on $\R$.

Let $A$, $\bar{A}$ be unbounded linear operators on $L^2(I)$:
\begin{equation*}
\begin{cases}
\mathcal{D}(A):=H^2(I)\cap H_0^1(I),\\
Av :=  v_{xx},\ \forall v\in\mathcal{D}(A),
\end{cases}
\end{equation*}
and
\begin{equation*}
\begin{cases}
\mathcal{D}(\bar{A}):=\{v\in H_\alpha^1(I):(x^{\alpha}v_x)_x\in L^2(I) ~\text{and}~ BC_\alpha(v)=0 \},\\
\bar{A}v :=  (x^{\alpha}v_x)_x,\ \forall v\in\mathcal{D}(\bar{A}),\,\alpha\in(0,2),
\end{cases}
\end{equation*}
where
$$
\begin{array}{ll}
H^1_\alpha(I):=\bigg\{v\in L^2(I):v \text{ is absolutely continuous in}~ I,
 x^{\frac{\alpha}{2}}v_x\in L^2(I)\,\text{and}~\,v(1)=0 \bigg\},
\end{array}
$$
and
\begin{equation*}
BC_\alpha(v)=
\begin{cases}
v_{|_{x=0}}, &\alpha\in[0,1),\\
(x^{\alpha}v_x)_{|_{x=0}}, &\alpha\in[1,2). 
\end{cases}
\end{equation*}
The first purpose of this paper is to study the controllability of the following linear coupled system:
\begin{equation}
\label{eq:main1}
\begin{cases}
y_t = A y + ay+ bz+ \chi_E\chi_{G_1} u, & \left( x,t\right) \in I\times(0,T),   \\[2mm]
z_t = \bar{A} z + cy+ dz+ \chi_E\chi_{G_2} u, & \left( x,t\right) \in I\times(0,T),   \\[2mm]
y\left(x, 0\right) =y_{0},z\left(x, 0\right) =z_{0}, &  x\in I,\\[2mm]
(0<\alpha<1)\
\begin{cases}
y(1,t)=z(1,t) = 0, & t\in \left(0,T\right), \\[2mm]
y(0,t)=z(0,t) = 0, & t\in \left(0,T\right),
\end{cases}\\[2mm]
(1\leq \alpha<2)\
\begin{cases}
y(1,t)=z(1,t) = 0, & t\in \left(0,T\right), \\[2mm]
(x^\alpha z_x)(0,t)=y(0,t) = 0, & t\in \left(0,T\right),
\end{cases}
\end{cases}
\end{equation}
here, $u\in L^\infty(0,T;L^2(G_1\cup G_2))$ is the control, $(y, z)$ is the state variable, $(y_0, z_0) \in \big( L^2(I) \big)^2$ is any given initial value, $a,b,c,d\in L^\infty(0,T;\R)$, $E$ be measurable subset with positive measure of $[0, T]$. By \cite{evans2022partial} and \cite{cannarsa2005}, one can check that for any $u\in L^\infty(0,T;L^2(G_1\cup G_2))$, systems (\ref{eq:main1}) and (\ref{eq:main}) admit a unique solution $(y, z)$ in the class of
$$
(y, z)\in \big(C(0,T;L^2(I))\big)^2\cap \big(L^2(0,T;H_0^1(I))\times L^2(0,T;H_\alpha^1(I))\big).
$$
\begin{theorem}
  \label{thm:main1}
Let $T>0$, $\alpha\in(0,2)$, $\sigma>0$ (which will be defined later, see Lemma \ref{lemma-A2}). Supposed $E$ is measurable subset of $[0,T]$ with positive measure. The coupled system \eqref{eq:main1} is null controllable. That is,
for each initial data $(y_0,z_0)\in \big(L^2(I)\big)^2$, there is a control $u$ in the space $L^\infty(0,T;L^2(G_1\cup G_2))$ such that
the solution $(y,z)$ of the coupled system \eqref{eq:main1} satisfies $y(T) =z(T)= 0$
in $I$. Moreover,there is a constant $C=C(T,I,\alpha,\sigma, |E|,G_1,G_2)>0$ and the control $u$ satisfies the following
estimate
\begin{equation}
  \label{eq:control est1}
  \|u\|_{L^\infty(0,T;L^2(G_1\cup G_2))} \le C  (\|y_0\|+\|z_0\|).
\end{equation}
\end{theorem}

The second purpose of this paper is to study the controllability of the following linear coupled system:
\begin{equation}
\label{eq:main}
\begin{cases}
y_t = A y + ay+ bz+ \chi_E\chi_{G_1} u, & \left( x,t\right) \in I\times(0,T),   \\[2mm]
z_t = \bar{A} z + cy+ dz+ \chi_F\chi_{G_2} u, & \left( x,t\right) \in I\times(0,T),   \\[2mm]
y\left(x, 0\right) =y_{0},z\left(x, 0\right) =z_{0}, &  x\in I,\\[2mm]
(0<\alpha<1)\
\begin{cases}
y(1,t)=z(1,t) = 0, & t\in \left(0,T\right), \\[2mm]
y(0,t)=z(0,t) = 0, & t\in \left(0,T\right),
\end{cases}\\[2mm]
(1\leq \alpha<2)\
\begin{cases}
y(1,t)=z(1,t) = 0, & t\in \left(0,T\right), \\[2mm]
(x^\alpha z_x)(0,t)=y(0,t) = 0, & t\in \left(0,T\right),
\end{cases}
\end{cases}
\end{equation}
here, for all $i\in\N^+$, $E=\cup_{i=1}^\infty E_i$, $E_i=(t_{2i-1},t_{2i})$, $F=\cup_{i=1}^\infty F_i$, $F_i=(t_{2i},t_{2i+1})$, such that $(0,T)=E\cup F$ and $E\cap F=\emptyset$. $u\in L^2(0,T;L^2(G_1\cup G_2))$ is the control, $(y, z)$ is the state variable, $(y_0, z_0) \in \big( L^2(I) \big)^2$ is any given initial value, $a,b,c,d\in L^\infty(0,T;\R)$. By \cite{evans2022partial} and \cite{cannarsa2005}, one can check that for any $u\in L^2(0,T;L^2(G_1\cup G_2))$, systems (\ref{eq:main1}) and (\ref{eq:main}) admit a unique solution $(y, z)$ in the class of
$$
(y, z)\in \big(C(0,T;L^2(I))\big)^2\cap \big(L^2(0,T;H_0^1(I))\times L^2(0,T;H_\alpha^1(I))\big).
$$

We suppose the following conditions:
\begin{itemize}
  \item [($H_1$)] There exists a positive constant $l_0$ and an open interval $E_{i_0}, i_0\in\N^+$, such that either
$c(t) \geq l_0$ for all $t \in E_{i_0}$, or $c(t)\leq -l_0$ for all $t\in E_{i_0}$.
  \item [($H_2$)] There exists a positive constant $\bar{l}_0$ and an open interval  $F_{i_0}, i_0\in\N^+$, such that either
$b(t) \geq \bar l_0$ for all $t \in F_{i_0}$, or $b(t)\leq -\bar l_0$ for all $t\in F_{i_0}$.
\end{itemize}
\begin{theorem}
  \label{thm:main}
Let $T>0$, $\alpha\in(0,2)$, $\sigma>0$ (which will be defined later, see Lemma \ref{lemma-A2}). If the conditions $(H_1)$ or $(H_2)$ hold, the coupled system \eqref{eq:main} is null controllable. That is, for given initial value $(y_0,z_0)\in \big(L^2(I)\big)^2$, there is a control $u\in L^2(0,T;L^2(G_1\cup G_2))$ such that
the solution $(y,z)$ of the coupled system \eqref{eq:main} satisfies
$y(T) =z(T) =0, ~\text{in}~ I$.
Moreover, there is a constant $L>0$, such that the control $u$ satisfies the following estimate
\begin{equation}
  \label{eq:control est}
  \|u\|^2_{L^2(0,T;L^2(G_1\cup G_2))} \le L  (\|y_0\|^2+\|z_0\|^2).
\end{equation}
\end{theorem}

On the other hand, we have the following negative result for the null controllability of coupled system (\ref{eq:main}).
\begin{theorem}
  \label{thm:negative-main}
The coupled system (\ref{eq:main}) is not null controllable at time $T$ provided that one of the following conditions is satisfied:
\begin{itemize}
  \item [(1)] $\chi_E\equiv1$, $c(\cdot)=0$ in $(0,T)$, a.e.;
  \item [(2)] $\chi_F\equiv1$, $b(\cdot)=0$ in $(0,T)$, a.e.
\end{itemize}
\end{theorem}
Next, we study the following adjoint system of coupled systems (\ref{eq:main1}) and (\ref{eq:main}):
\begin{equation}
\label{eq:adjoint}
\begin{cases}
p_t = -\bar{A} p - dp- bw, & \left( x,t\right) \in I\times(0,T),   \\[2mm]
w_t = -A w -cp-aw, & \left( x,t\right) \in I\times(0,T),   \\[2mm]
p\left(x, T\right) =p_T,w\left(x, T\right) =w_T, &  x\in I,\\[2mm]
(0<\alpha<1)\
\begin{cases}
p(1,t)=w(1,t) = 0, & t\in \left(0,T\right), \\[2mm]
p(0,t)=w(0,t) = 0, & t\in \left(0,T\right),
\end{cases}\\[2mm]
(1\leq \alpha<2)\
\begin{cases}
p(1,t)=w(1,t) = 0, & t\in \left(0,T\right), \\[2mm]
(x^\alpha p_x)(0,t)=w(0,t) = 0, & t\in \left(0,T\right),
\end{cases}
\end{cases}
\end{equation}
where $(p_T,w_T)\in \big(L^2(I)\big)^2$ is any given terminal value.
\begin{theorem}
  \label{thm:obs-inq1}
Let $T>0$, $\alpha\in(0,2)$, $\sigma>0$ (which will be defined later, see Lemma \ref{lemma-A2}). Supposed $E$ is measurable subset of $[0,T]$ with positive measure. Then there exists a constant $C=C(T,I,\alpha,\sigma, |E|,G_1,G_2)$ such that the following observability inequality holds: for any $(p_T,w_T)\in \big(L^2(I)\big)^2$,
\begin{equation}
  \label{eq:obs-inq1} 
  \|p(0)\|^2+\|w(0)\|^2
\le C\left[\|\chi_E\chi_{G_1}w\|^2_{L^1(0,T;L^2(I))}+\|\chi_E\chi_{G_2}p\|^2_{L^1(0,T;L^2(I))}\right]. 
\end{equation}
\end{theorem}
\begin{theorem}
  \label{thm:obs-inq}
Let $T>0$, $\sigma>0$ (which will be defined later, see Lemma \ref{lemma-A2}). Suppose  $E=\cup_{i=1}^\infty E_i$, $E_i=(t_{2i-1},t_{2i})$, $F=\cup_{i=1}^\infty F_i$, $F_i=(t_{2i},t_{2i+1})$, such that $(0,T)=F\cup E$ and $E\cap F=\emptyset$. Then there exists a constant $C>0$ such that the following observability inequality holds: for any $(p_T,w_T)\in \big(L^2(I)\big)^2$,
\begin{equation}
  \label{eq:obs-inq} 
  \|p(0)\|^2+\|w(0)\|^2
\le C\left[\|\chi_E\chi_{G_1}w\|^2_{L^2(0,T;L^2(I))}+\|\chi_F\chi_{G_2}p\|^2_{L^2(0,T;L^2(I))}\right]. 
\end{equation}
\end{theorem}
Several remarks are given in order.
\begin{remark}
The original problem that we want to consider is the following problem: 
\begin{equation}\label{eq:origin}
    \begin{cases}
        y_t = A y + ay+ bz+ \chi_E\chi_{G_1} u, & \left( x,t\right) \in I\times(0,T),   \\[2mm]
z_t = \bar{A} z + cy+ dz+ \chi_F\chi_{G_2} u, & \left( x,t\right) \in I\times(0,T),   \\[2mm]
y\left(x, 0\right) =y_{0},z\left(x, 0\right) =z_{0}, &  x\in I,\\[2mm]
(0<\alpha<1)\
\begin{cases}
y(1,t)=z(1,t) = 0, & t\in \left(0,T\right), \\[2mm]
y(0,t)=z(0,t) = 0, & t\in \left(0,T\right),
\end{cases}\\[2mm]
(1\leq \alpha<2)\
\begin{cases}
y(1,t)=z(1,t) = 0, & t\in \left(0,T\right), \\[2mm]
(x^\alpha z_x)(0,t)=y(0,t) = 0, & t\in \left(0,T\right),
\end{cases}
    \end{cases}
\end{equation}
where $E, F\subset (0,T)$ are {positive measurable subsets} with $E\cup F=(0,T)$ and $E\cap F=\emptyset$, and the control $u\in L^2(0,T;L^2(I))$. It is clearly that \eqref{eq:origin} is a kind of switching control problem. Unfortunately, we do not solve this problem, but we get two special cases for this problem, i.e., the first problem \eqref{eq:main1} and the second problem \eqref{eq:main}. 
For system \eqref{eq:main1}, we have considered switching controls $u$ on the same measurable  set in time, in this case, we consider a more strong control problem with control $u$ belongs to space $L^\infty(0,T;L^2(I))$, then we obtain the controllability of the system \eqref{eq:main1}. For system \eqref{eq:main}, we have obtained switching controls $u\in L^2(0,T;L^2(I))$ on segmented time intervals.
\end{remark}

\begin{remark}
We derive observability inequality for measurable subsets in time directly from the adjoint system of system (\ref{eq:main1}). In our analysis, we employ spectral inequalities as the main tool, as indicated in the Lemmas \ref{lemma-A1} and \ref{lemma-A2}, which will be defined later, instead of Carleman estimates. The reason for this is that the construction of suitable Carleman weight functions, which is the main technique employed in Carleman inequalities, appears to be infeasible, as demonstrated in \cite{apraiz2013null}. Notably, an intriguing issue arises as the control regions $G_1$ and $G_2$ are also measurable subsets rather than open subsets. Unfortunately, we have yet to obtain a similar spectral inequality for this problem, as we must consider the degeneracy of the operator $\bar{A}$. For other observability inequalities related to measurable sets, see \cite{phung2013observability,apraiz2014observability,escauriaza2015observation}  and references therein.
\end{remark}
\begin{remark}
As for system (\ref{eq:main}), we utilize the Lebeau-Robbiano strategy, as outlined in \cite{miller2009direct,wang2008null} and references therein, based on spectral inequalities, see Lemmas \ref{lemma-A1} and \ref{lemma-A2}, which will be defined later, to establish the null controllability of the system under switching control for segmented time intervals and subsequently derive observability inequality. It is worth noting that the Theorem \ref{thm:obs-inq} implies special cases where either $E=\emptyset$ or $F=\emptyset$, in which only one control force is active in $G_2$ or $G_1$ for the coupled system (\ref{eq:main}).
\end{remark}

The following Sections of this paper are structured as follows. In Section 2, we present several supporting results. In Section 3, we establish observability inequality for the adjoint system corresponding to the coupled system (\ref{eq:main1}). More specifically, in Section 3.1, we present some observability results, and subsequently prove the observability inequality in Section 3.2. In Section 4, we investigate the null controllability of the adjoint system corresponding to the coupled system (\ref{eq:main}). In Section 4.1, we first provide additional observability results and then prove the null controllability in Section 4.2. Finally, we discuss the scenario of negative null controllability in Section 4.3 and provide the relevant observability inequality.

\section{Notations and Auxiliary Conclusions}
To begin with, we write
\[
0<\lambda_1\le \lambda_2\le \cdots
\]
for the eigenvalues of $-A$ with the zero Dirichlet boundary condition over $\partial I$, and $\{e_i\}_{i\ge 1}$ be the corresponding eigenfunctions such that $\|e_i\|=1$ for $i=1,2,\cdots$.

Write
\[
0<\bar{\lambda}_1\le \bar{\lambda}_2\le \cdots
\]
for the eigenvalues of $-\bar{A}$ with the boundary condition $BC_\alpha(v)=0$ over $\partial I$, and $\{\bar{e}_i\}_{i\ge 1}$ be the corresponding eigenfunctions such that $\|\bar{e}_i\|=1$ for $i=1,2,\cdots$.

Then for each $v\in D(A)$, we have
$$
v=\sum_{i=1}^\infty a_ie_i \mbox{ with } a_i\in \R,\ i\geq 1,\quad (-A)v=\sum_{i=1}^\infty \lambda_ia_ie_i,
$$
and define
\[
\cE_\lambda v = \sum_{\lambda_i\leq \lambda} a_ie_i,~\text{and}~
\cE^\perp_\lambda v = \sum_{\lambda_i>\lambda} a_ie_i.
\]
For each $\bar v\in D(\bar A)$, we have
$$
\bar{v}=\sum_{i=1}^\infty \bar{a}_i\bar{e}_i \mbox{ with } \bar{a}_i\in \R,\ i\geq 1,\quad (-\bar{A})\bar{v}=\sum_{i=1}^\infty \bar{\lambda}_i\bar{a}_i\bar{e}_i,
$$
and define
\[
\cE_{\bar{\lambda}} \bar{v} = \sum_{\bar\lambda_i\leq \bar\lambda} \bar{a}_i\bar{e}_i,~\text{and}~
\cE^\perp_{\bar{\lambda}} \bar{v} = \sum_{\bar\lambda_i>\bar\lambda} \bar{a}_i\bar{e}_i.
\]

For any positive integer $k$, set
$$
\begin{array}{lll}
X_k=\text{span}~\{e_1,e_2,\cdots,e_k\}\\[3mm]
\bar{X}_k=\text{span}~\{\bar{e}_1,\bar{e}_2,\cdots,\bar{e}_k\}.
\end{array}
$$
Denote by $\Pi_k$ the orthogonal projection from $L^2(I)$ to $X_k$, and $\bar{\Pi}_k$ the orthogonal projection from $L^2(I)$ to $\bar{X}_k$.

Next, we recall some known results. By \cite{buffe2018spectral}, one can get the spectral asymptotic formulas for the operator $A$ and $\bar{A}$.
\begin{lemma}
\label{lemma-lambda}
Set $\Omega$ be a bounded domain in $\R^d\ (d\geq 1)$,  there exists positive constants $C_1=C_1(\Omega)$ and $C_2=C_2(d)$, such
that the eigenvalues $\{\tilde{\lambda}_i\}_{i=1}^\infty$ of $\tilde{A}$ satisfy the following formula:
$$
\tilde{\lambda}_k\sim C_1k^{C_2},~\text{as}~ k\rightarrow\infty.
$$
\begin{itemize}
  \item If $\tilde{A}=-A=-\Delta$, then $C_2=\frac{2}{d}$,
  \item If $\tilde{A}=-\bar{A}=-\frac{\partial}{\partial x}\bigg(x^\alpha\frac{\partial}{\partial x}\bigg)$ with $\alpha\in (0,2)$, then $C_2=2$.
\end{itemize}
\end{lemma}
The other known lemmas are related to estimates for the partial sum of eigenfunctions of $A$ and $\bar{A}$, respectively; see \cite{lebeau1995controle,buffe2018spectral}.
\begin{lemma}
\label{lemma-A1}
If $G_0$ is a nonempty open subset in $I$, then there exists a positive constant $C_1=C_1(I,G_0)$, such that for any positive integer $k$ and any
numbers $a_i\in\R \ (i = 1, 2,\cdots, k)$, it holds that
\begin{equation}
\label{lemma-A1-eq}
    \sum_{i=1}^k|a_i|^2\leq C_1e^{C_1\sqrt{\lambda_k}}\int_{G_0}\bigg|\sum_{i=1}^ka_ie_i(x)\bigg|^2dx.
\end{equation}

\end{lemma}

\begin{lemma}
\label{lemma-A2}
If $G_0$ is a nonempty open subset in $I$, then there exists a positive constant $C_2=C_2(I,G_0,\alpha,\sigma)$, such that for any positive integer $k$ and any
numbers $\bar{a}_i\in\R\ (i = 1, 2,\cdots, k)$, it holds that
\begin{equation}
\label{lemma-A2-eq}
\sum_{i=1}^k|\bar{a}_i|^2\leq C_2e^{C_2\bar{\lambda}_k^\sigma}\int_{G_0}\bigg|\sum_{i=1}^k\bar{a}_i\bar{e}_i(x)\bigg|^2dx,
\end{equation}
where
\begin{equation}
\label{sigma}
 \sigma=
\begin{cases}
\dfrac{3}{4}, &~\text{if}~ \alpha\in(0,2)\setminus\{1\},\\[3mm]
\dfrac{3}{2\gamma}  ~\text{for any}~\gamma\in(0,2),&~\text{if}~ \alpha=1.
\end{cases}
\end{equation}

\end{lemma}
Then, (\ref{lemma-A1-eq}) and (\ref{lemma-A2-eq}) can be replaced by the following expressions, respectively.
\begin{equation}
  \label{eq:spectral-1}
\|\cE_\lambda \xi\|^2
\leq C_1e^{C_1\sqrt{\lambda_k}}\|\cE_\lambda \xi \|^2_{L^2(G_0)},\  \forall\, \xi\in
L^2(I), \ \lambda>0,
\end{equation}
and
\begin{equation}
  \label{eq:spectral-2}
\|\cE_{\bar{\lambda}} \xi\|^2
\leq C_2e^{C_2\bar{\lambda}_k^\sigma}\|\cE_{\bar{\lambda}} \xi \|^2_{L^2(G_0)},\  \forall\, \xi\in
L^2(I), \ \bar\lambda>0.
\end{equation}

Let us denote by $(p(\cdot),w(\cdot))$ the solution of coupled system
\eqref{eq:adjoint} given the terminal condition $(p_T,w_T)=(p(T),w(T))$.
Denote
$$
\cE_{\bar{\lambda}} p(t)=\sum_{\bar\lambda_i\leq \bar\lambda} p_i(t)\bar{e}_i,\quad\ \,\,\, \cE_\lambda w(t)=\sum_{\lambda_i\leq\lambda}w_i(t)e_i,
$$
and
\begin{equation}
\label{eq66}
\begin{array}{lll}
\cE^\perp_{\bar{\lambda}}p(t)=\sum_{\bar\lambda_i>\bar\lambda}p_i(t)\bar{e}_i,\,\,\,\,\cE^\perp_\lambda w(t)=\sum_{\lambda_i>\lambda} w_i(t)e_i,
\end{array}
\end{equation}
where $(p_i,w_i)\ (i\in\mathbb{N}^+)$ satisfies the following coupled system:
\begin{equation}
\label{model-i}
\begin{cases}
(p_i)_t = \bar{\lambda}_i p_i - dp_i- bw_i, & \left( x,t\right) \in I\times(0,T),   \\[2mm]
(w_i)_t = \lambda_i w_i -cp_i-aw_i, & \left( x,t\right) \in I\times(0,T),   \\[2mm]
p_i\left(x, T\right) =p_{T,i},w_i\left(x, T\right) =w_{T,i}, &  x\in I\,,\\[2mm]
(0<\alpha<1)\
\begin{cases}
p_i(1,t)=w_i(1,t) = 0, & t\in \left(0,T\right), \\[2mm]
p_i(0,t)=w_i(0,t) = 0, & t\in \left(0,T\right),
\end{cases}\\[2mm]
(1\leq \alpha<2)\
\begin{cases}
p_i(1,t)=w_i(1,t) = 0, & t\in \left(0,T\right), \\[2mm]
(x^\alpha (p_i)_x)(0,t)=w_i(0,t) = 0, & t\in \left(0,T\right).
\end{cases}
\end{cases}
\end{equation}
For simplifying the notations, write
$$
L^\infty:=L^\infty(0,T;\R).
$$
Set $$\tau = 2\max\{\|a\|_{L^\infty},\|d\|_{L^\infty}\}+\|b\|_{L^\infty}+\|c\|_{L^\infty}+1.$$

\section{Observability estimate and  Null controllability for measurable sets in time}

\subsection{ Some Observability results}

In this subsection, we denote $\bar \lambda=\bar\lambda_k$ and $\lambda=\lambda_k$ with some $k\in\mathbb{N}^+$ for simplifying the notations. 

\begin{lemma}
  For any $\alpha\in(0,2)$ and given any $(p_T,w_T)$ in the space of  $\big(L^2(I)\big)^2$,
    we have for each $t\in [0,T]$,
  \begin{equation}
    \label{eq:decay}
     \|\cE^\perp_{\bar{\lambda}}p(t)\|^2+\|\cE^\perp_{\lambda}w(t)\|^2 \le e^{\left(-2\min\{\bar{\lambda}_k,\lambda_k\}+\tau\right)(T-t)}  (\|p_T\|^2+\|w_T\|^2).
  \end{equation}
\end{lemma}
\begin{proof}
 At first, we have (by the definition of $\tau$, we have $\tau>0$)
\begin{equation*}
    \begin{split}
        \big( e^{(2\bar{\lambda}_k-\tau)(T- t)}[\cE^\perp_{\bar{\lambda}}p(t)]^2 \big)_t
        &=-(2\bar{\lambda}_k-\tau)e^{(2\bar{\lambda}_k-\tau)(T- t)}[\cE^\perp_{\bar{\lambda}}p(t)]^2\\
&\hspace{3.5mm}+2e^{(2\bar{\lambda}_k-\tau)(T- t)} [\cE^\perp_{\bar{\lambda}}p(t)]\big\{-\bar{A}[\cE^\perp_{\bar{\lambda}}p(t)]-d[\cE^\perp_{\bar{\lambda}}p(t)]-b[\cE^\perp_\lambda w(t)]\big\}, 
    \end{split}
\end{equation*}
and
\begin{equation*}
    \begin{split}
        \big( e^{(2\lambda_k-\tau)(T- t)}[\cE^\perp_{\lambda}w(t)]^2 \big)_t
&=-(2\lambda_k-\tau)e^{(2\lambda_k-\tau)(T- t)}[\cE^\perp_\lambda w(t)]^2\\
&\hspace{3.5mm}+2e^{(2\lambda_k-\tau)(T- t)} [\cE^\perp_{\lambda}w(t)]\big\{-A[\cE^\perp_{\lambda}w(t)]
-c[\cE^\perp_{\bar{\lambda}}p(t)]-a[\cE^\perp_\lambda w(t)]\big\}. 
    \end{split}
\end{equation*}
Then for above equalities, integrating from $(0,1)\times(t,T)$, respectively, it follows from equality (\ref{eq66}) that
\begin{equation*}
    \begin{split}
        &\|\cE^\perp_{\bar{\lambda}} p_T\|^2-e^{(2\bar{\lambda}_k-\tau)(T- t)}\|\cE^\perp_{\bar{\lambda}}p(t)\|^2\\
&=\int_t^Te^{(2\bar{\lambda}_k-\tau)(T- s)}\big\{ 2\langle \cE^\perp_{\bar{\lambda}}p(s),-\bar{A}[\cE^\perp_{\bar{\lambda}}p(s)]\rangle-2d\langle \cE^\perp_{\bar{\lambda}}p(s),\cE^\perp_{\bar{\lambda}}p(s)\rangle
-2b\langle \cE^\perp_{\bar{\lambda}}p(s),\cE^\perp_{\lambda}w(s)\rangle\big\}ds\\
&\hspace{3.5mm}-\int_t^Te^{(2\bar{\lambda}_k-\tau)(T- s)}(2\bar{\lambda}_k-\tau)\|\cE^\perp_{\bar{\lambda}}p(s)\|^2 ds\\
&=\int_t^Te^{(2\bar{\lambda}_k-\tau)(T- s)}\bigg[2\sum_{i=k+1}^\infty \bar{\lambda}_i
  |p_i(s)|^2-2\bar{\lambda}_k\|\cE^\perp_{\bar{\lambda}}p(s)\|^2-2d\langle \cE^\perp_{\bar{\lambda}}p(s),\cE^\perp_{\bar{\lambda}}p(s)\rangle\\
  &
\hspace{3cm}-2b\langle \cE^\perp_{\bar{\lambda}}p(s),\cE^\perp_{\lambda}w(s)\rangle+\tau\|\cE^\perp_{\bar{\lambda}}p(s)\|^2\bigg]ds\\
&\geq\int_t^Te^{(2\bar{\lambda}_k-\tau)(T- s)}\bigg[-2\|d\|_{L^\infty}\| \cE^\perp_{\bar{\lambda}}p(s)\|^2
-\|b\|_{L^\infty}\|\cE^\perp_{\bar{\lambda}}p(s)\|^2\\
&\hspace{3cm}-\|b\|_{L^\infty}\| \cE^\perp_{\lambda}w(s)\|^2+\tau\|\cE^\perp_{\bar{\lambda}}p(s)\|^2\bigg]ds,
    \end{split}
\end{equation*}
and
\begin{equation*}
\begin{split} 
&\|\cE_\lambda^\bot w_T\|^2-e^{(2\lambda_k-\tau)(T- t)}\|\cE_\lambda^\bot w(t)\|^2\\
&=\int_t^Te^{(2\lambda_k-\tau)(T- s)}\big\{ 2\langle \cE_\lambda^\bot w( s),-A[\cE_\lambda^\bot w(s)]\rangle-2a\langle \cE_\lambda^\bot w(s),\cE_\lambda^\bot w(s)\rangle
-2c\langle \cE_\lambda^\bot w(s),\cE^\perp_{\bar{\lambda}}p(s)\rangle\big\}ds\\
&\hspace{3.5mm}-\int_t^Te^{(2\lambda_k-\tau)(T- s)}(2\lambda_k-\tau)\|\cE_\lambda^\bot w(s)\|^2 ds\\
&=\int_t^Te^{(2\lambda_k-\tau)(T- s)}\bigg[2\sum_{i=k+1}^\infty\lambda_i
  |w_i(s)|^2-2\lambda_k\|\cE_\lambda^\bot w(s)\|^2-2a\langle \cE_\lambda^\bot w(s),\cE_\lambda^\bot w(s)\rangle\\
  &
\hspace{3cm}-2c\langle \cE^\perp_{\bar{\lambda}}p(s),\cE_\lambda^\bot w(s)\rangle+\tau\|\cE_\lambda^\bot w(s)\|^2\bigg]ds\\
&\geq\int_t^Te^{(2\lambda_k-\tau)(T- s)}\bigg[-2\|a\|_{L^\infty}\|\cE_\lambda^\bot w(s)\|^2
-\|c\|_{L^\infty}\|\cE^\perp_{\bar{\lambda}} p(s)\|^2\\
&\hspace{3cm}-\|c\|_{L^\infty}\|\cE_\lambda^\bot w(s)\|^2+\tau\|\cE_\lambda^\bot w(s)\|^2\bigg]ds.
\end{split}
\end{equation*}
From these, we see
\begin{equation*}
    \begin{split}
        &\|\cE^\perp_{\bar{\lambda}} p_T\|^2+\|\cE_\lambda^\bot w_T\|^2-e^{(2\bar{\lambda}_k-\tau)(T- t)}\|\cE^\perp_{\bar{\lambda}}p(t)\|^2-e^{(2\lambda_k-\tau)(T- t)}\|\cE_\lambda^\bot w(t)\|^2\\
&\geq\int_t^Te^{\left(2\min\{\lambda_k,\bar{\lambda}_k\}-\tau\right)(T- s)}\bigg\{ \big[ \tau-(2\|d\|_{L^\infty}-\|b\|_{L^\infty}-\|c\|_{L^\infty}) \big]\|\cE^\perp_{\bar{\lambda}}p(s)\|^2\\
&\hspace{4.2cm}+\big[ \tau-(2\|a\|_{L^\infty}-\|b\|_{L^\infty}-\|c\|_{L^\infty}) \big]\|\cE_\lambda^\bot w(s)\|^2  \bigg\}ds\\
&\geq\int_t^Te^{\left(2\min\{\lambda_k,\bar{\lambda}_k\}-\tau\right)(T- s)}\bigg\{ \big[ \tau-\big(2\max\{\|a\|_{L^\infty},\|d\|_{L^\infty}\}\\
&\hspace{5cm}-\|b\|_{L^\infty}-\|c\|_{L^\infty}\big) \big]\big(\|\cE^\perp_{\bar{\lambda}}p(s)\|^2+\|\cE_\lambda^\bot w(s)\|^2\big)\bigg\}ds\\
&\geq0, 
    \end{split}
\end{equation*}
which implies the inequality \eqref{eq:decay}.
\end{proof}

Next, we provide an interpolation inequality.
\begin{proposition}
  Let $\alpha\in(0,2)$, $\sigma$ be defined in \eqref{sigma}. Given any $(p_T,w_T)$ in the space of  $\big(L^2(I)\big)^2$, and $t\in [0,T)$, there
  exists
 a constant $K = K(T,I,G_1,G_2,\alpha,\sigma)$ such that
\begin{equation}
  \label{eq:interpolation} 
   \|p(t)\|^2+\|w(t)\|^2\le Ke^{K(T-t)^{\frac{\sigma}{\sigma-1}}} \big(\|p(t)\|^2_{L^2(G_2)}+\|w(t)\|^2_{L^2(G_1)}\big)^{\frac{1}{2}}\big( \|p_T\|^2+\|w_T\|^2\big)^{\frac{1}{2}}. 
\end{equation}
\end{proposition}

\begin{proof}
On one hand, it follows from the
  spectral estimate
  \eqref{eq:spectral-2} that
\begin{align*}
   \|\cE_{\bar{\lambda}} p(t)\|^2
\le C_2e^{C_2\bar{\lambda}_k^\sigma} \|\cE_{\bar{\lambda}} p(t)\|^2_{L^2(G_2)}\le C_2e^{C_2\bar{\lambda}_k^\sigma}\big( \|p(t)\|^2_{L^2(G_2)} +
\|\cE^\perp_{\bar{\lambda}} p(t)\|^2_{L^2(G_2)}\big)
\end{align*}
for some constant $C_2 = C_2(I,G_2,\alpha,\sigma)$. It shows that
\begin{align*}
\|p(t)\|^2
 &=  \|\cE_{\bar{\lambda}} p(t)\|^2+\|\cE^\perp_{\bar{\lambda}}
p(t)\|^2\\
& \le  C_2e^{C_2\bar{\lambda}_k^\sigma}\big( \|p(t)\|^2_{L^2(G_2)} +
\|\cE^\perp_{\bar{\lambda}} p(t)\|^2_{L^2(G_2)}\big)+ \|\cE^\perp_{\bar{\lambda}} p(t)\|^2\\
& \le 2C_2e^{C_2\bar{\lambda}_k^\sigma}\big( \| p(t)\|^2_{L^2(G_2)} + \|\cE^\perp_{\bar{\lambda}} p(t)\|^2\big).
\end{align*}
On the other hand, it follows from the
  spectral estimate
  \eqref{eq:spectral-1} that
\begin{align*}
   \|\cE_\lambda w(t)\|^2\le C_1e^{C_1\sqrt{\lambda_k}} \|\cE_\lambda w(t)\|^2_{L^2(G_1)}\le C_1e^{C_1\sqrt{\lambda_k}}\big( \|w(t)\|^2_{L^2(G_1)} +
\|\cE^\bot_\lambda w(t)\|^2_{L^2(G_1)}\big)
\end{align*}
for some constant $C_1 = C_1(I,G_1)$. It shows that
\begin{align*}
\|w(t)\|^2
 &=  \|\cE_\lambda w(t)\|^2+\|\cE^\bot_\lambda
w(t)\|^2\\
& \le  C_1e^{C_1\sqrt{\lambda_k}}\big( \|w(t)\|^2_{L^2(G_1)} +
\|\cE^\bot_\lambda w(t)\|^2_{L^2(G_1)}\big)+ \|\cE^\bot_\lambda w(t)\|^2\\
& \le 2C_1e^{C_1\sqrt{\lambda_k}}\big( \| w(t)\|^2_{L^2(G_1)} + \|\cE^\bot_\lambda w(t)\|^2\big).
\end{align*}
Therefore, for some constant $C = C(I,G_1,G_2,\alpha,\sigma)$, by the decay estimate
\eqref{eq:decay} we obtain that
\begin{equation*}
    \begin{split}
        &\|p(t)\|^2+\|w(t)\|^2\\
        &\le 2C_2e^{C_2\bar{\lambda}_k^\sigma}\big( \| p(t)\|^2_{L^2(G_2)} + \|\cE^\perp_{\bar{\lambda}} p(t)\|^2\big)+2C_1e^{C_1\sqrt{\lambda_k}}\big( \| w(t)\|^2_{L^2(G_1)} + \|\cE^\bot_\lambda w(t)\|^2\big)\\
&\le 2Ce^{C(\bar{\lambda}_k^\sigma+\sqrt{\lambda_k})}\big( \| p(t)\|^2_{L^2(G_2)}+\| w(t)\|^2_{L^2(G_1)} +
\|\cE^\perp_{\bar{\lambda}} p(t)\|^2+\|\cE^\bot_\lambda w(t)\|^2\big)\\
&\le 2Ce^{C(\bar{\lambda}_k^\sigma+\sqrt{\lambda_k})}\big[ \| p(t)\|^2_{L^2(G_2)}+\| w(t)\|^2_{L^2(G_1)} +e^{(-2\min\{\bar{\lambda}_k,\lambda_k\}+\tau)(T-t)}  (\|p_T\|^2+\|w_T\|^2)\big]\\
&\le  2Ce^{C(\bar{\lambda}_k^\sigma+\sqrt{\lambda_k})}e^{\tau T}\big[\| p(t)\|^2_{L^2(G_2)}+\| w(t)\|^2_{L^2(G_1)} +e^{-2\min\{\bar{\lambda}_k,\lambda_k\}(T-t)} (\|p_T\|^2+\|w_T\|^2)\big]\\
& =2Ce^{C(\bar{\lambda}_k^\sigma+\sqrt{\lambda_k})-\min\{\bar{\lambda}_k,\lambda_k\}(T-t)}e^{\tau T}\big[e^{\min\{\bar{\lambda}_k,\lambda_k\}(T-t)}(\| p(t)\|^2_{L^2(G_2)}+\| w(t)\|^2_{L^2(G_1)})\\
&\hspace{5.5cm}+
e^{-\min\{\bar{\lambda}_k,\lambda_k\}(T-t)} (\|p_T\|^2+\|w_T\|^2)\big].
    \end{split}
\end{equation*} 
By Lemma \ref{lemma-lambda} and (\ref{sigma}),
one can verify that for some constants $K_1=K_1(I,G_1,G_2,\alpha,\sigma)$,
\begin{equation*}
    \begin{split} 
 C(\bar{\lambda}_k^\sigma+\sqrt{\lambda_k})-\min\{\bar{\lambda}_k,\lambda_k\}(T-t)
 &\le C_3(k^{2\sigma}+k)-C_4k^2(T-t)\\
 &\le C_3k^{2\sigma}-C_4k^2(T-t)
 \le K_1(T-t)^{\frac{\sigma}{\sigma-1}}, 
    \end{split}
\end{equation*}
in the last inequality, we have taken maximum about the variable $k\in\mathbb{N}^+$. 
Hence, there exists a constant $K=K(T,I,G_1,G_2,\alpha,\sigma)$ such that
\begin{equation*}
    \begin{split}
         \|p(t)\|^2+\|w(t)\|^2
 &\le K e^{K(T-t)^{\frac{\sigma}{\sigma-1}}}\big[e^{\min\{\bar{\lambda}_k,\lambda_k\}(T-t)}(\| p(t)\|^2_{L^2(G_2)}+\| w(t)\|^2_{L^2(G_1)})\\
 &\hspace{2.5cm}+
e^{-\min\{\bar{\lambda}_k,\lambda_k\}(T-t)} (\|p_T\|^2+\|w_T\|^2)\big], 
    \end{split}
\end{equation*}
which is equivalent to:  for all $\e\in (0,1)$, we have 
\begin{equation}
  \label{eq:epsilon inq} 
  \|p(t)\|^2+\|w(t)\|^2\le K e^{K(T-t)^{\frac{\sigma}{\sigma-1}}}\big[\e^{-1}(\| p(t)\|^2_{L^2(G_2)}+\| w(t)\|^2_{L^2(G_1)}) +
\e (\|p_T\|^2+\|w_T\|^2)\big]. 
\end{equation}
Noting that $\| p(t)\|^2+\| w(t)\|^2 \le C_0 (\|p_T\|^2+\|w_T\|^2)$, where $C_0$ is a
constant depending on $T$, we see that the inequality
\eqref{eq:epsilon inq} holds for all $\e>0$. Finally, minimizing
\eqref{eq:epsilon inq} with respect to $\e$ leads to the desired
estimate \eqref{eq:interpolation}.
\end{proof}
\subsection{Observability Inequality And Null Controllability}
In this subsection, we are ready to prove Theorem \ref{thm:obs-inq1}.
\begin{proof}[Proof of Theorem \ref{thm:obs-inq1}]
 We borrow some ideas from \cite{yang2016observability} and \cite{apraiz2014observability}. Set $E$ is a measurable subset of $[0,T]$ with positive measure. Let $\ell\in(0,T)$ be any Lebesgue point of $E$. Then for each constant $q\in(0,1)$ which is to be fixed later, there exists a monotone increasing sequence
$\{\ell_n\}_{n\geq1}$ in $(0, \ell)$ such that
$$\lim_{n\rightarrow +\infty}\ell_n=\ell,$$
\begin{equation}\label{eq:equiv ratio}
\ell_{n+2}-\ell_{n+1}=q(\ell_{n+1}-\ell_n),\;\forall\, n\geq1
\end{equation}
and
$$|E\cap (\ell_n,\ell_{n+1})|\geq \frac{\ell_{n+1}-\ell_n}{3},\;\forall\, n\geq1.$$
Set
\begin{equation}
\label{eq-tau}
\tau_n=\ell_{n+1}-\frac{\ell_{n+1}-\ell_n}{6},\;\forall\, n\geq1.
\end{equation}
For each $t\in (\ell_n, \tau_n)$, by the interpolation inequality
\eqref{eq:interpolation} and replacing $ T$ by $\ell_{n+1}$, we have
\begin{equation*}
    \begin{split}
         \|p(t)\|^2+\|w(t)\|^2
   \le Ke^{K(l_{n+1}-t)^{\frac{\sigma}{\sigma-1}}} \big(\|p(t)\|^2_{L^2(G_2)}+\|w(t)\|^2_{L^2(G_1)}\big)^{\frac{1}{2}}\big( \|p(l_{n+1})\|^2+\|w(l_{n+1})\|^2\big)^{\frac{1}{2}}. 
    \end{split}
\end{equation*}
Since
\[
\ell_{n+1} -t \ge \ell_{n+1} -\tau_n = \frac{\ell_{n+1} -\ell_n}{6},
\]
and for some constant $C=C(T)$, $ \|p(\ell_n)\|^2+\|w(\ell_n)\|^2 \le C
(\|p(t)\|^2+\|w(t)\|^2)$, there exists a constant $C = C(T,I,G_1,G_2, \alpha,\sigma)$ such
that for all $n\ge 1$, and $t\in (\ell_n, \tau_n)$,
\begin{equation*}
    \begin{split}
        \|p(\ell_n)\|^2+\|w(\ell_n)\|^2
   \le Ce^{C(\ell_{n+1}-\ell_n)^{\frac{\sigma}{\sigma-1}}} \big(\|p(t)\|^2_{L^2(G_2)}+\|w(t)\|^2_{L^2(G_1)}\big)^{\frac{1}{2}}\big( \|p(\ell_{n+1})\|^2+\|w(\ell_{n+1})\|^2\big)^{\frac{1}{2}}. 
    \end{split}
\end{equation*}
Using the Cauchy inequality with $\epsilon$, the above inequality leads to
\begin{equation*}
    \begin{split}
        \|p(\ell_n)\|^2+\|w(\ell_n)\|^2 
&\le \e^{-1} C e^{C(\ell_{n+1}-\ell_n)^{\frac{\sigma}{\sigma-1}}} (\|p(t)\|^2_{L^2(G_2)}+\|w(t)\|^2_{L^2(G_1)})\\
&\hspace{3.5mm}+ \e (\|p(\ell_{n+1})\|^2+\|w(\ell_{n+1})\|^2), 
    \end{split}
\end{equation*}
which implies
\begin{equation*}
    \begin{split}
        \big(\|p(\ell_n)\|^2+\|w(\ell_n)\|^2\big)^{\frac{1}{2}} 
&\le \e^{-1} C e^{C(\ell_{n+1}-\ell_n)^{\frac{\sigma}{\sigma-1}}} (\|p(t)\|_{L^2(G_2)}+\|w(t)\|_{L^2(G_1)})\\
&\hspace{3.5mm}+ \e (\|p(\ell_{n+1})\|^2+\|w(\ell_{n+1})\|^2)^{\frac{1}{2}}. 
    \end{split}
\end{equation*}
Equivalently, we have
\begin{equation}
  \label{eq:Am-inq}
  A_n \le \e^{-1} C e^{C(\ell_{n+1}-\ell_n)^{\frac{\sigma}{\sigma-1}}} B(t) + \e A_{n+1},
\end{equation}
where
\begin{equation}
  \label{eq:Am-B}
  A_n =\big(\|p(\ell_n)\|^2+\|w(\ell_n)\|^2\big)^{\frac{1}{2}},\quad B(t)=\|p(t)\|_{L^2(G_2)}+\|w(t)\|_{L^2(G_1)}.
\end{equation}
Noting that $\{\ell_n\}_{n\geq1}$ is a monotone increasing sequence
in $(0,\ell)$, it follows that
\begin{align*}
  |E\cap(\ell_n,\tau_n)|
& = |E \cap (\ell_n,\ell_{n+1})| - |E\cap (\tau_n,\ell_{n+1})|\ge \frac{\ell_{n+1}-\ell_{n}}{3} - \frac{\ell_{n+1}-\ell_{n}}{6}= \frac{\ell_{n+1}-\ell_{n}}{6}>0.
\end{align*}
Then integrating the inequality \eqref{eq:Am-inq} over $E\cap (\ell_n,\tau_n)$, we have that for each $\e>0$,
$$
\int_{E\cap (\ell_n,\tau_n)}A_ndt\leq \e^{-1}C e^{C(\ell_{n+1}-\ell_n)^{\frac{\sigma}{\sigma-1}}}\int_{\ell_n}^{\tau_n}\chi_EB(t)dt+\e\int_{E\cap (\ell_n,\tau_n)}A_{n+1}dt,
$$
which  implies
$$
A_n\leq \e^{-1}C e^{C(\ell_{n+1}-\ell_n)^{\frac{\sigma}{\sigma-1}}}|E\cap(\ell_n,\tau_n)|^{-1}\int_{\ell_n}^{\tau_n}\chi_EB(t)dt+\e A_{n+1}.
$$
By (\ref{eq-tau}) and $|E\cap(\ell_n,\tau_n)|\geq \frac{1}{6}(\ell_{n+1}-\ell_n), \sigma\neq 1$, it follows that
\[
A_n \le \e A_{n+1}+ \e^{-1} C e^{C(\ell_{n+1}-\ell_n)^{\frac{\sigma}{\sigma-1}}} \int_{\ell_{n}}^{\ell_{n+1} } \chi_E B(t) dt,
\]
where $C=C(I,G_1,G_2,T,\alpha,\sigma,|E|)$. 
Multiplying the above inequality by $\e e^{-C(\ell_{n+1}-\ell_n)^{\frac{\sigma}{\sigma-1}}}$,
and replacing $\e$ by $\sqrt{\e}$ lead to
\[
\sqrt{\e}  e^{-C(\ell_{n+1}-\ell_n)^{\frac{\sigma}{\sigma-1}}} A_n \le \e  e^{-C(\ell_{n+1}-\ell_n)^{\frac{\sigma}{\sigma-1}}} A_{n+1} + C \int_{\ell_{n}}^{\ell_{n+1} } \chi_E B(t) dt.
\]
 Finally choosing $\e =
e^{-(\ell_{n+1}-\ell_n)^{\frac{\sigma}{\sigma-1}}}$ in the above inequality, we get
\begin{align*}
  & e^{-(C+\frac{1}{2})(\ell_{n+1}-\ell_n)^{\frac{\sigma}{\sigma-1}}} A_n -
  e^{-(C+1)(\ell_{n+1}-\ell_n)^{\frac{\sigma}{\sigma-1}}} A_{n+1}\le C \int_{\ell_{n}}^{\ell_{n+1} } \chi_E B(t) dt.
\end{align*}
Now, choosing $q = \bigg(\frac{C+\frac{1}{2}}{C+1}\bigg)^{\frac{1-\sigma}{\sigma}}$ in \eqref{eq:equiv ratio}, we
have
\begin{align*}
  & e^{-(C+\frac{1}{2})(\ell_{n+1}-\ell_n)^{\frac{\sigma}{\sigma-1}}} A_n -
  e^{-(C+\frac{1}{2})(\ell_{n+2}-\ell_{n+1})^{\frac{\sigma}{\sigma-1}}} A_{n+1} \le C \int_{\ell_{n}}^{\ell_{n+1} } \chi_E B(t) dt.
\end{align*}
Summing the above inequality from $n=1$ to $+\infty$, we have
\[
A_1 \le C e^{(C+\frac{1}{2})(\ell_2-\ell_1)^{\frac{\sigma}{\sigma-1}}}
\int_{\ell_1}^\ell \chi_E B(t) dt.
\]
Plugging the substitution \eqref{eq:Am-B} into the above inequality, we obtain
\begin{equation*}
    \begin{split}
         \|p(\ell_1)\|^2+\|w(\ell_1)\|^2
  &\le C e^{(C+\frac{1}{2})(\ell_2-\ell_1)^{\frac{\sigma}{\sigma-1}}}
\left(\int_{\ell_1}^\ell \chi_E (\|p(t)\|_{L^2(G_2)}+\|w(t)\|_{L^2(G_1)}) dt\right)^2\\
&\le C e^{(C+\frac{1}{2})(\ell_2-\ell_1)^{\frac{\sigma}{\sigma-1}}}\left[\|\chi_E\chi_{G_1}w\|^2_{L^1(0,T;L^2(I))}+\|\chi_E\chi_{G_2}p\|^2_{L^1(0,T;L^2(I))}
\right], 
    \end{split}
\end{equation*}
which implies the observability inequality \eqref{eq:obs-inq1},
completing the proof.
\end{proof}
Next, by the standard duality augment (i.e., HUM), we have the following equivalence between
the null controllability of (\ref{eq:main1}) and the observability inequality (\ref{eq:obs-inq1}) for the adjoint equation
(\ref{eq:adjoint}).

\begin{proposition}
  \label{pro-equivalence}
For any $T>0$, $\alpha\in(0,2)$ and $\sigma\in I$, the coupled system (\ref{eq:main1}) is null controllable at
  time $T$ with the control $u$ in the space of
  $L^\infty(0,T;L^2(G_1\cup G_2)$ such that the estimate (\ref{eq:control est1}) holds if and only if there exists a constant $C=C(T,I,\alpha,\sigma, |E|,G_1,G_2)$ such that the
  solution of the coupled system (\ref{eq:adjoint}) satisfies the
  observability inequality (\ref{eq:obs-inq1}).
\end{proposition}
Then Theorem \ref{thm:main1} is a direct consequence of Theorem \ref{thm:obs-inq1} and Proposition \ref{pro-equivalence}.

\section{Observability estimate and  Null controllability for segmented time intervals}

\subsection{ Some Observability Results}
At first, some observability results for (\ref{eq:adjoint}) can be stated as follows.
\begin{proposition}
\label{pro-H1}
Suppose that the condition $(H_1)$ holds. Then there exists a positive constant $C$, such that for any terminal value
$(p_T, w_T ) \in L^2(\bar{X}_k)\times L^2(X_k)$, the corresponding solution $(p,w)$ of (\ref{eq:adjoint}) satisfies
$$
 \|p(0)\|^2+\|w(0)\|^2 \le \dfrac{C(\bar\lambda_k^2+\lambda_k^2)e^{C\sqrt{\lambda_k}+\tau T}}{T}\int_E\int_{G_1}w^2(x,t)dxdt.
$$

\end{proposition}

\begin{proof}
we divide the proof into following parts.

{\it Step 1}.
Each $(p_T,w_T)\in L^2(\bar{X}_k)\times L^2(X_k)$ can be written as
$$
 p_T=\sum_{i=1}^kp_{T,i}\bar{e}_i,\quad  w_T=\sum_{i=1}^kw_{T,i}e_i,
$$
where $(p_{T,i},w_{T,i})\in \R^2, i=1,\cdots, k$.
Then solution $(p,w)$ to (\ref{eq:adjoint}) can be expressed as
\begin{equation}
\label{pro-eq0}
 p(t)=\sum_{i=1}^kp_i(t)\bar{e}_i,\,\,  w(t)=\sum_{i=1}^kw_i(t)e_i,
\end{equation}
where $(p_i,w_i)$ for $i=1,2,\cdots,k$ satisfies the coupled system (\ref{model-i}).

{\it Step 2}. We give an estimate on $\|p(0)\|^2+\|w(0)\|^2$, where $p$ and $w$ are in \eqref{pro-eq0}.

At first, we have (by the definition of $\tau$ and $(H_1)$, we have $\tau>0$)
\begin{equation*}
    \begin{split}
        \big( e^{\tau t}p^2 \big)_t
        &=\tau e^{\tau t}p^2+2e^{\tau t} pp_t=\tau e^{\tau t}p^2+2e^{\tau t} p(-\bar{A}p-dp-bw),\\
\big( e^{\tau t}w^2 \big)_t
&=\tau e^{\tau t}w^2+2e^{\tau t} ww_t=\tau e^{\tau t}w^2+2e^{\tau t} w(-Aw-cp-aw). 
    \end{split}
\end{equation*}
Then for above equalities, integrating from $(0,1)\times(0,t)$, respectively, it holds
\begin{equation*}
    \begin{split}
        &e^{\tau t}\int_0^1p^2(x,t)dx-\int_0^1p^2(x,0)dx\\
&=\int_0^t\int_0^1e^{\tau s}\big\{ 2p[-\bar{A}p-dp-bw]+\tau p^2 \big\}dxds\\
&=2\int_0^te^{\tau s}\sum_{i=1}^k\bar{\lambda}_i(p_i(s))^2ds+\int_0^t\int_0^1e^{\tau s}\big[ -2dp^2-2bpw+\tau p^2 \big]dxds\\
&\geq2k\bar{\lambda}_1\int_0^t\int_0^1e^{\tau s}p^2(x,s)dxds+\tau \int_0^t\int_0^1e^{\tau s}p^2dxds\\
&\hspace{3.5mm}-\int_0^t\int_0^1e^{\tau s}\big[ 2\|d\|_{L^\infty}p^2+\|b\|_{L^\infty}p^2+\|b\|_{L^\infty}w^2 \big]dxds,
    \end{split}
\end{equation*}
and
\begin{equation*}
\begin{split}
&e^{\tau t}\int_0^1w^2(x,t)dx-\int_0^1w^2(x,0)dx\\[3mm]
&=\int_0^t\int_0^1e^{\tau s}\left\{ 2w[-Aw-cp-aw]+\tau w^2 \right\}dxds\\
&=2\int_0^te^{\tau s}\sum_{i=1}^k\lambda_i(w_i(s))^2ds+\int_0^t\int_0^1e^{\tau s}\big[ -2aw^2-2cpw+\tau w^2 \big]dxds\\
&\geq2k\lambda_1\int_0^t\int_0^1e^{\tau s}w^2(x,s)dxds+\tau \int_0^t\int_0^1e^{\tau s}w^2dxds\\
&\hspace{3.5mm}-\int_0^t\int_0^1e^{\tau s}\big[ 2\|a\|_{L^\infty}p^2+\|c\|_{L^\infty(0,T)}p^2+\|c\|_{L^\infty}w^2 \big]dxds.
\end{split}
\end{equation*}
From these, we see
\begin{equation*}
\begin{split}
    &e^{\tau t}\int_0^1p^2(x,t)dx-\int_0^1p^2(x,0)dx+e^{\tau t}\int_0^1w^2(x,t)dx-\int_0^1w^2(x,0)dx\\
&\geq2k\min\{\lambda_1,\bar{\lambda}_1\}\int_0^t\int_0^1e^{\tau s}[p^2(x,s)+w^2(x,s)]dxds+\tau \int_0^t\int_0^1e^{\tau s}(p^2+w^2)dxds\\
&\hspace{3.5mm}-\int_0^t\int_0^1e^{\tau s}\left[ (2\|d\|_{L^\infty}+\|b\|_{L^\infty}+\|c\|_{L^\infty})p^2+(2\|a\|_{L^\infty}+\|b\|_{L^\infty}+\|c\|_{L^\infty})w^2 \right]dxds\\
&\geq0 
\end{split}
\end{equation*}
by the definition of $\tau$,
which implies
\begin{equation}
\label{pro-eq1}
\int_0^1p^2(x,0)dx+\int_0^1w^2(x,0)dx\leq e^{\tau t}\int_0^1p^2(x,t)dx+ e^{\tau t}\int_0^1w^2(x,t)dx.
\end{equation}

{\it Step 3.} We establish a local estimate for $p^2$.

To the end, we borrow some ideas from \cite{liu2014controllability}. For fixed $i_0\in\N^+$, let $(s_1,s_2)\subseteq E_{i_0}$ such that $s_2-s_1=\dfrac{T}{C_0}$ for some positive constant $C_0$. Choose a cutoff function $\xi\in C^\infty_0(E_{i_0})$ such that $0\leq\xi\leq1$, $\xi=1$ in $(s_1,s_2)$ and $\dfrac{\xi_t^2}{\xi}\in L^\infty(0,T;\R)$. By (\ref{eq:adjoint}), we have
\begin{equation*}
    \begin{split}
       (\xi pw)_t&=\xi_tpw+\xi pw_t+\xi p_tw\\
&=\xi_tpw+\xi p\big[ -Aw-cp-aw \big]+\xi w\big[ -\bar{A}p-dp-bw \big].
    \end{split}
\end{equation*}
Then for above equality, integrating from $I\times E_{i_0}$, using Cauchy inequality with $\epsilon$ and H\"{o}lder inequality, it holds
\begin{equation*}
    \begin{split}
        &\left|\int_{E_{i_0}}\int_0^1c(t)\xi p^2dxdt\right|\\
&=\int_{E_{i_0}}\int_0^1\xi_tpwdxdt+\int_{E_{i_0}}\xi\sum_{i=1}^k\bar{\lambda}_ip_i(t)w_i(t)dt+\int_{E_{i_0}}\xi\sum_{i=1}^k\lambda_ip_i(t)w_i(t)dt\\
&\hspace{3.5mm}+\int_{E_{i_0}}\int_0^1\xi\big[ -apw-dpw-bw^2 \big]dxdt\\
&\leq\epsilon\int_{E_{i_0}}\int_0^1\xi p^2dxdt+\frac{C}{\epsilon}\int_{E_{i_0}}\int_0^1\frac{\xi_t^2}{\xi} w^2dxdt+(\bar{\lambda}_k+\lambda_k)\int_{E_{i_0}}\xi\|p(\cdot,t)\|\|w(\cdot,t)\|dt\\
&\hspace{3.5mm}+\epsilon\big( \|a\|_{L^\infty}+\|d\|_{L^\infty} \big)\int_{E_{i_0}}\int_0^1\xi p^2dxdt+\frac{C}{\epsilon}\int_{E_{i_0}}\int_0^1\xi w^2dxdt\\
&\leq \epsilon(1+\|a\|_{L^\infty}+\|d\|_{L^\infty})\int_{E_{i_0}}\int_0^1\xi p^2dxdt+C(\bar\lambda_k^2+\lambda_k^2)\int_{E_{i_0}}\int_0^1 w^2dxdt,
    \end{split}
\end{equation*}
here and what in follows, $C>0$ represent different constants (independent on $k$) by different context.
Without loss of generality, we suppose that in the condition $(H_1)$, $c(t) \geq l_0 > 0$ in $E_{i_0}$. The above inequality follows that
\begin{equation*}
\begin{split}
    &l_0\int_{E_{i_0}}\int_0^1\xi p^2dxdt\\
    &\leq\left|\int_{E_{i_0}}\int_0^1c\xi p^2dxdt\right|\\
    &\leq \epsilon(1+\|a\|_{L^\infty}+\|b\|_{L^\infty})\int_{E_{i_0}}\int_0^1\xi p^2dxdt + C(\bar\lambda_k^2+\lambda_k^2)\int_{E_{i_0}}\int_0^1 w^2dxdt,
\end{split}
\end{equation*}
taking $\epsilon(1+\|a\|_{L^\infty}+\|b\|_{L^\infty})=\frac{l_0}{2}$,
which implies
$$
\int_{E_{i_0}}\int_0^1\xi p^2dxdt\leq C(\bar\lambda_k^2+\lambda_k^2)\int_{E_{i_0}}\int_0^1w^2dxdt.
$$
This, alone with $\xi=1$ in $(s_1,s_2)$, shows
\begin{equation}
\label{pro-eq2}
\int_{s_1}^{s_2}\int_0^1p^2dxdt\leq\int_{E_{i_0}}\int_0^1\xi p^2dxdt\leq C(\bar\lambda_k^2+\lambda_k^2)\int_{E_{i_0}}\int_0^1 w^2dxdt.
\end{equation}

{\it Step 4}. Integrating (\ref{pro-eq1}) on $(s_1, s_2)$ with respect to the variable $t$, we have
\begin{equation}
\label{pro-eq3}
\begin{split}
&\int_0^1p^2(x,0)dx+\int_0^1w^2(x,0)dx\\
&\leq\frac{e^{\tau T}}{s_1-s_2}\int_{s_1}^{s_2}\int_0^1\big[p^2(x,t)+w^2(x,t)\big]dxdt\\
&\leq\frac{Ce^{\tau T}}{T}\int_{s_1}^{s_2}\int_0^1p^2(x,t)dxdt+\frac{Ce^{\tau T}}{T}\int_{s_1}^{s_2}\int_0^1w^2(x,t)dxdt.
\end{split}
\end{equation}
By Lemma \ref{lemma-A1} and (\ref{pro-eq0}), we see that for a.e. $t \in (0, T)$,
\begin{equation*}
    \begin{split}
        \int_0^1w^2(x,t)dx=\sum_{i=1}^k\big( w_i(t) \big)^2\leq Ce^{C\sqrt{\lambda_k}}\int_{G_1}\bigg|\sum_{i=1}^kw_i(t)e_i(x)\bigg|^2dx
=Ce^{C\sqrt{\lambda_k}}\int_{G_1}w^2(x,t)dx, 
    \end{split}
\end{equation*}
which, together with (\ref{pro-eq2}) and (\ref{pro-eq3}), it stands
\begin{equation*}
    \begin{split}
        &\int_0^1p^2(x,0)dx+\int_0^1w^2(x,0)dx\\
        &\leq\frac{C(\bar\lambda_k^2+\lambda_k^2)e^{\tau T}}{T}\int_{E_{i_0}}\int_0^1w^2(x,t)dxdt
+\frac{Ce^{\tau T}}{T}\int_{s_1}^{s_2}\int_0^1w^2(x,t)dxdt\\
&\leq\frac{C(\bar\lambda_k^2+\lambda_k^2)e^{C\sqrt{\lambda_k}+\tau T}}{T}\int_{E_{i_0}}\int_{G_1}w^2(x,t)dxdt.
    \end{split}
\end{equation*}
The proof is completed.
\end{proof}

Similar to the proof of Proposition \ref{pro-H1}, we have the following conclusion.
\begin{proposition}
\label{pro-H2}
Suppose that the condition $(H_2)$ holds. Then there exists a positive constant $C$, such that for any terminal value
$(p_T, w_T ) \in L^2(\bar{X}_k)\times L^2(X_k)$, the corresponding solution $(p,w)$ of (\ref{eq:adjoint}) satisfies
$$
 \|p(0)\|^2+\|w(0)\|^2 \le \dfrac{C(\bar\lambda_k^2+\lambda_k^2)e^{C\bar\lambda_k^\sigma+\tau T}}{T}\int_F\int_{G_2}p^2(x,t)dxdt.
$$
\end{proposition}
By means of the usual duality argument, Proposition \ref{pro-H1} and Proposition \ref{pro-H2} yields a partial controllability result for coupled system (\ref{eq:main}).
\begin{proposition}
\label{pro-partial}
Suppose that the condition $(H_1)$ or $(H_2)$ holds. Then for any positive integer $k$, one can find a control $u^{(k)}\in L^2(0,T;L^2(G_1\cup G_2))$ such that
the corresponding solution $(y, z)$ of (\ref{eq:main}) satisfies
\begin{equation}
\label{pro-partial-eq0}
\Pi_k(y(T))=\bar{\Pi}_k(z(T))=0,~\text{in}~ I.
\end{equation}
Moreover, there exists a positive constant $C$ so that
\begin{equation}
\label{pro-partial-eq00}
\|y(T)\|^2+\|z(T)\|^2\leq \bigg(\dfrac{C(\bar\lambda_k^2+\lambda_k^2)e^{\tau T}(e^{C\bar{\lambda}_k^\sigma}+e^{C\sqrt{\lambda_k}})}{T}+1\bigg)e^{\tau T}(\|y(0)\|^2+\|z(0)\|^2),
\end{equation}
and
\begin{equation}
\label{pro-partial-eq000}
\|u^{(k)}\|^2_{L^2(0,T;L^2(G_1\cup G_2))}\leq \dfrac{C(\bar\lambda_k^2+\lambda_k^2)e^{\tau T}(e^{C\bar{\lambda}_k^\sigma}+e^{C\sqrt{\lambda_k}})}{T}(\|y(0)\|^2+\|z(0)\|^2).
\end{equation}
\end{proposition}

\begin{proof}
We divide the proof into following steps.

{\it Step 1.} Set  $E=\cup_{i=1}^\infty E_i$, $E_i=(t_{2i-1},t_{2i})$, $F=\cup_{i=1}^\infty F_i$, $F_i=(t_{2i},t_{2i+1})$, such that $(0,T)=F\cup E$ and $E\cap F=\emptyset$. We introduce a linear subspace of $L^2(0,T;L^2(G_1\cup G_2))$ :
$$
H=\bigg\{ \chi_E\chi_{G_1} w+\chi_F\chi_{G_2}p:(p,w) ~\text{solves}~ (\ref{eq:adjoint}) ~\text{with some}~ (p_T, w_T) \in L^2(\bar{X}_k)\times L^2(X_k)\bigg\}.
$$
Define a linear functional on $H$ as follows:
\begin{equation}
\label{pro-partial-H1-eq1}
\cL(\chi_E\chi_{G_1} w+\chi_F\chi_{G_2}p)=-\int_0^1z_0(x)p(x,0)dx-\int_0^1y_0(x)w(x,0)dx,
\end{equation}
where $(y_0, z_0)$ is any given initial value of coupled system (\ref{eq:main}). By Proposition \ref{pro-H1} and Proposition \ref{pro-H2}, it follows that
\begin{equation*}
    \begin{split}
        \|p(0)\|^2
        &\leq \dfrac{C(\bar\lambda_k^2+\lambda_k^2)e^{\tau T}(e^{C\bar{\lambda}_k^\sigma}+e^{C\sqrt{\lambda_k}})}{T}\bigg(\int_E\int_{G_1}w^2(x,t)dxdt+\int_F\int_{G_2}p^2(x,t)dxdt\bigg), \\
\|w(0)\|^2
&\leq \dfrac{C(\bar\lambda_k^2+\lambda_k^2)e^{\tau T}(e^{C\bar{\lambda}_k^\sigma}+e^{C\sqrt{\lambda_k}})}{T}\bigg(\int_E\int_{G_1}w^2(x,t)dxdt+\int_F\int_{G_2}p^2(x,t)dxdt\bigg).
    \end{split}
\end{equation*}
Since
\begin{equation*}
    \begin{split}
        \bigg|\int_0^1z_0(x)p(x,0)dx\bigg|
        \leq \|z_0\|\|p(0)\|,\quad
\bigg|\int_0^1y_0(x)w(x,0)dx\bigg|
\leq \|y_0\|\|w(0)\|,
    \end{split}
\end{equation*}
from these, alone with (\ref{pro-partial-H1-eq1}), it holds
\begin{equation*}
    \begin{split}
       |\cL(\chi_E\chi_{G_1} w+\chi_F\chi_{G_2}p)|
       &\leq\dfrac{C(\bar\lambda_k^2+\lambda_k^2)e^{\frac{\tau T}{2}}(e^{C\bar{\lambda}_k^\sigma}+e^{C\sqrt{\lambda_k}})}{\sqrt T}\big[ \|y_0\|+\|z_0\| \big]\\
&\hspace{3.5mm}\cdot\bigg( \int_E\int_{G_1}w^2(x,t)dxdt+\int_F\int_{G_2}p^2(x,t)dxdt \bigg)^{\frac{1}{2}}.  
    \end{split}
\end{equation*}
This implies that $\cL$ is a bounded linear functional on $H$ and the norm of $\mathcal{L}$ is
$$
\|\cL\|_{H^*}\leq \dfrac{C(\bar\lambda_k^2+\lambda_k^2)e^{\frac{\tau T}{2}}(e^{C\bar{\lambda}_k^\sigma}+e^{C\sqrt{\lambda_k}})}{\sqrt T}\big[ \|y_0\|+\|z_0\| \big].
$$
Then, by the Hahn-Banach Theorem, $\cL$ can be extended to a bounded linear functional $\tilde{\cL}$ on $L^2(0,T;L^2(G_1\cup G_2))$ and
$$
\|\tilde{\cL}\|_{L(L^2(0,T;L^2(G_1\cup G_2));\R)}=\|\cL\|_{H^*}.
$$
Therefore, combining with (\ref{pro-partial-H1-eq1}), one can find a $u^{(k)}\in L^2(0,T;L^2(G_1\cup G_2))$ such that
\begin{equation}
\label{pro-partial-H1-eq2} 
\begin{split}
\int_E\int_{G_1}u^{(k)}wdxdt+\int_F\int_{G_2}u^{(k)}pdxdt
&=\tilde{\cL}(\chi_E\chi_{G_1} w+\chi_F\chi_{G_2}p)\\
&=-\int_0^1z_0(x)p(x,0)dx-\int_0^1y_0(x)w(x,0)dx, 
\end{split}
\end{equation}
and
\begin{equation}
\label{pro-partial-H1-eq3}
\begin{split}
\|u^{(k)}\|_{L^2(0,T;L^2(G_1\cup G_2))}
&=\|\tilde{\cL}\|_{L(L^2(0,T;L^2(G_1\cup G_2));\R)}\\
&\leq \dfrac{C(\bar\lambda_k^2+\lambda_k^2)e^{\frac{\tau T}{2}}(e^{C\bar{\lambda}_k^\sigma}+e^{C\sqrt{\lambda_k}})}{\sqrt T}\big[ \|y_0\|+\|z_0\| \big].
\end{split}
\end{equation}

{\it Step 2.} We prove that $u^{(k)}$ is the desired control.

Indeed, by (\ref{eq:main}) with $u=u^{(k)}$ and (\ref{eq:adjoint}),  since 
$$
(yw)_t=yw_t+wy_t,\quad\ (zp)_t=zp_t+pz_t,
$$
for any $(p_T, w_T ) \in L^2(\bar{X}_k)\times L^2(X_k)$, integrating the above equalities over $(0,1) \times(0,T)$, respectively, it holds
\begin{equation*}
    \begin{split}
        &\int_E\int_{G_1}u^{(k)}wdxdt+\int_F\int_{G_2}u^{(k)}pdxdt\\
        &=\int_0^1y(x,T)w_T(x)dx+\int_0^1z(x,T)p_T(x)dx-\int_0^1z_0(x)p(x,0)dx-\int_0^1y_0(x)w(x,0)dx, 
    \end{split}
\end{equation*}
which, alone with (\ref{pro-partial-H1-eq2}), it implies (\ref{pro-partial-eq0}). Then (\ref{pro-partial-H1-eq3}) stands  the desired estimate (\ref{pro-partial-eq000}).

{\it Step 3.} We give the desired estimate for the terminal value $(y(T ), z(T ))$.

Since $\tau>0$, we have
\begin{equation*}
\begin{split}
\big( e^{-\tau t}y^2 \big)_t
&=-\tau e^{-\tau t}y^2+2e^{-\tau t} y(Ay+ay+bz+\chi_E\chi_{G_1}u^{(k)})\\
\big( e^{-\tau t}z^2 \big)_t
&=-\tau e^{-\tau t}z^2+2e^{-\tau t} z(\bar{A}z+cy+dz+\chi_F\chi_{G_2}u^{(k)}).
\end{split}
\end{equation*}
Then for above equailties, integrating from $(0,1)\times(0,T)$, respectively and using H\"{o}lder inequality, it holds
\begin{equation*}
    \begin{split}
        &e^{-\tau T}\int_0^1y^2(x,T)dx-\int_0^1y_0^2dx\\
&=\int_0^T\int_0^1e^{-\tau t}\big\{ 2y[Ay+ay+bz+\chi_E\chi_{G_1}u^{(k)}]-\tau y^2 \big\}dxdt\\
&\leq\int_0^T\int_0^1e^{-\tau t}\big[ 2\|a\|_{L^\infty}y^2+\|b\|_{L^\infty}(y^2+z^2)+y^2+\chi_E\chi_{G_1}(u^{(k)})^2-\tau y^2 \big]dxdt-2k\lambda_1\int_0^T\int_0^1e^{-\tau t}y^2dxdt,
    \end{split}
\end{equation*}
and
\begin{equation*}
    \begin{split}
        &e^{-\tau T}\int_0^1z^2(x,T)dx-\int_0^1z_0^2dx\\
&=\int_0^T\int_0^1e^{-\tau t}\big\{ 2z[\bar{A}z+cy+dz+\chi_F\chi_{G_2}u^{(k)}]-\tau z^2 \big\}dxdt\\
&\leq\int_0^T\int_0^1e^{-\tau t}\big[ 2\|d\|_{L^\infty}z^2+\|c\|_{L^\infty}(y^2+z^2)+z^2+\chi_F\chi_{G_2}(u^{(k)})^2-\tau z^2 \big]dxdt-2k\bar{\lambda}_1\int_0^T\int_0^1e^{-\tau t}z^2dxdt.
    \end{split}
\end{equation*}
From these, we see
\begin{equation*}
\begin{split}
&e^{-\tau T}\int_0^1y^2(x,T)dx+e^{-\tau T}\int_0^1z^2(x,T)dx-\int_0^1y_0^2dx-\int_0^1z_0^2dx\\
&\leq-2k\min\{\lambda_1,\bar{\lambda}_1\}\int_0^T\int_0^1e^{-\tau t}[y^2(x,t)+z^2(x,t)]dxdt\\
&\hspace{3.5mm}+\int_0^T\int_0^1e^{-\tau t}\big[ 2\|a\|_{L^\infty}y^2+2\|d\|_{L^\infty}z^2+(\|b\|_{L^\infty}+\|c\|_{L^\infty}+1)(y^2+z^2) \big]dxdt\\
&\hspace{3.5mm}-\tau \int_0^T\int_0^1e^{-\tau t}(y^2+z^2)dxdt+\int_E\int_0^1e^{-\tau t}\chi_{G_1}(u^{(k)})^2dxdt+\int_F\int_0^1e^{-\tau t}\chi_{G_2}(u^{(k)})^2dxdt\\
&\leq-2k\min\{\lambda_1,\bar{\lambda}_1\}\int_0^T\int_0^1e^{-\tau t}[y^2(x,t)+z^2(x,t)]dxdt\\
&\hspace{3.5mm}+\tau \int_0^T\int_0^1e^{-\tau t}(y^2+z^2)dxdt-\tau \int_0^T\int_0^1e^{-\tau t}(y^2+z^2)dxdt\\
&\hspace{3.5mm}+\int_E\int_0^1e^{-\tau t}\chi_{G_1}(u^{(k)})^2dxdt+\int_F\int_0^1e^{-\tau t}\chi_{G_2}(u^{(k)})^2dxdt\\
&\leq C \bigg(\int_E\int_0^1\chi_{G_1}(u^{(k)})^2dxdt+\int_F\int_0^1\chi_{G_2}(u^{(k)})^2dxdt\bigg)\\
&\leq C\int_0^T\int_{G_1\cup G_2}(u^{(k)})^2dxdt.
\end{split}
\end{equation*}
From these, it holds
\begin{equation*}
    \begin{split}
        &e^{-\tau T}\int_0^1y^2(x,T)dx+e^{-\tau T}\int_0^1z^2(x,T)dx\leq\int_0^1y_0^2dx+\int_0^1z_0^2dx+ C\int_0^T\int_{G_1\cup G_2}(u^{(k)})^2dxdt.
    \end{split}
\end{equation*}
This, together with (\ref{pro-partial-H1-eq3}), we get the desired estimate (\ref{pro-partial-eq00}). The proof is completed.
\end{proof}

Next, we give a decay estimate for solutions of the coupled system (\ref{eq:main}).
\begin{proposition}
\label{pro-decay}
For any positive integer $k$, if $u\equiv0$ in coupled system (\ref{eq:main}), then for any $(y_0,z_0)\in\big( L^2(I) \big)^2$ satisfying $\Pi_k(y_0)=\bar{\Pi}_k(z_0)=0~\text{in}~ I$, the corresponding solution $(y, z)$ of (\ref{eq:main}) satisfies: for all $t\in[0,T]$,
\begin{equation}
\label{pro-decay-eq0}
\|y(t)\|^2+\|z(t)\|^2\leq e^{-(2\min\{ \lambda_{k+1},\bar{\lambda}_{k+1} \}-\tau)t}(\|y(0)\|^2+\|z(0)\|^2).
\end{equation}

\end{proposition}

\begin{proof}
For any $(y_0,z_0)\in\big( L^2(I) \big)^2$ satisfying $\Pi_k(y_0)=\bar{\Pi}_k(z_0)=0~\text{in}~ I$, write
$$
y_0=\sum_{i=k+1}^\infty y_{0,i}e_i\quad\ ~\text{and}~ \quad\ z_0=\sum_{i=k+1}^\infty z_{0,i}\bar{e}_i,
$$
where $y_{0,i},z_{0,i}$, $i=k+1,k+2,\cdots$, are real numbers. Then the solution $(y, z)$ of (\ref{eq:main}) can be represented as
$$
y=\sum_{i=k+1}^\infty y_i(t)e_i\quad\ ~\text{and}~ \quad\ z=\sum_{i=k+1}^\infty z_i(t)\bar{e}_i,
$$
where $(y_i,z_i)$ satisfies the following system:
$$
\begin{cases}
(y_i)_t = -\lambda_i y_i + ay_i+ bz_i, & \left( x,t\right) \in I\times(0,T),   \\[2mm]
(z_i)_t = -\bar{\lambda}_i z_i + cy_i+ dz_i, & \left( x,t\right) \in I\times(0,T),   \\[2mm]
y_i\left(x, 0\right) =y_{0,i},z_i\left(x, 0\right) =z_{0,i}, &  x\in I,\\[2mm]
(0<\alpha<1)\
\begin{cases}
y_i(1,t)=z_i(1,t) = 0, & t\in \left(0,T\right), \\[2mm]
y_i(0,t)=z_i(0,t) = 0, & t\in \left(0,T\right),
\end{cases}\\[2mm]
(1\leq \alpha<2)\
\begin{cases}
y_i(1,t)=z_i(1,t) = 0, & t\in \left(0,T\right), \\[2mm]
(x^\alpha (z_{i})_x)(0,t)=y_i(0,t) = 0, & t\in \left(0,T\right),
\end{cases}
\end{cases}
$$
Let $$\beta\triangleq2\min\{ \lambda_{k+1},\bar{\lambda}_{k+1} \}-\tau.$$
Since
\begin{equation*}
    \begin{split}
        \big( e^{\beta t}y^2 \big)_t
        &=\beta e^{\beta t}y^2+2e^{\beta t} yy_t=\beta e^{\beta t}y^2+2e^{\beta t} y(Ay+ay+bz),\\
\big( e^{\beta t}z^2 \big)_t
&=\beta e^{\beta t}z^2+2e^{\beta t} zz_t=\beta e^{\beta t}z^2+2e^{\beta t} z(\bar{A}z+cy+dz). 
    \end{split}
\end{equation*}
Then for any $t\in(0,T)$, integrating the above equalities from $(0,1)\times(0,t)$, respectively, it holds
\begin{equation*}
    \begin{split}
        &e^{\beta t}\int_0^1y^2(x,t)dx-\int_0^1y_0^2dx\\
&=\int_0^te^{\beta s}\sum_{i=k+1}^\infty (-2\lambda_i) (y_i(s))^2ds+\beta\int_0^t\int_0^1e^{\beta s}y^2(x,s)dxds
+\int_0^t\int_0^1e^{\beta s}2y[ay+bz]dxds,
    \end{split}
\end{equation*}
and
\begin{equation*}
    \begin{split}
        &e^{\beta t}\int_0^1z^2(x,t)dx-\int_0^1z_0^2dx\\
&=\int_0^te^{\beta s}\sum_{i=k+1}^\infty (-2\bar{\lambda}_i) (z_i(s))^2ds+\beta\int_0^t\int_0^1e^{\beta s}z^2(x,s)dxds
+\int_0^t\int_0^1e^{\beta s}2z[cy+dz]dxds.
    \end{split}
\end{equation*}
From these, using H\"{o}lder inequality we see
\begin{equation*}
\begin{split}
&e^{\beta t}\int_0^1y^2(x,t)dx+e^{\beta t}\int_0^1z^2(x,t)dx-\int_0^1y_0^2dx-\int_0^1z_0^2dx\\
&=\int_0^te^{\beta s}\bigg\{\sum_{i=k+1}^\infty \big[(-2\lambda_i) (y_i(s))^2+ (-2\bar{\lambda}_i) (z_i(s))^2\big]\bigg\}ds+\beta\int_0^t\int_0^1e^{\beta s}(y^2+z^2)dxds\\
&\hspace{3.5mm}+\int_0^t\int_0^1e^{\beta s}[2ay^2+2[b+c]yz+2dz^2]dxds\\
&\leq\int_0^te^{\beta s}\bigg\{\sum_{i=k+1}^\infty \big( -2\min\{ \lambda_{k+1},\bar{\lambda}_{k+1} \} \big)[(y_i(s))^2+(z_i(s))^2]\bigg\}ds
+\beta\int_0^t\int_0^1e^{\beta s}(y^2+z^2)dxds\\
&\hspace{3.5mm}+\int_0^t\int_0^1e^{\beta s}[2ay^2+2[b+c]yz+2dz^2]dxds\\
&\leq\big( -2\min\{ \lambda_{k+1},\bar{\lambda}_{k+1}\}\big) \int_0^t\int_0^1e^{\beta s} (y^2+z^2)dxds+\beta\int_0^t\int_0^1e^{\beta s}(y^2+z^2)dxds\\
&\hspace{3.5mm}+\int_0^t\int_0^1e^{\beta s}\big[ 2\|a\|_{L^\infty}y^2+(\|b\|_{L^\infty}+\|c\|_{L^\infty})(y^2+z^2)+2\|d\|_{L^\infty}z^2 \big]dxds\\
&\leq\big\{ \beta+\big[\big( -2\min\{ \lambda_{k+1},\bar{\lambda}_{k+1}\}\big)+\tau\big] \big\}\int_0^t\int_0^1e^{\beta s}(y^2+z^2)dxds\\
&\leq0,
\end{split}
\end{equation*}
which implies the inequality (\ref{pro-decay-eq0}).
\end{proof}

\subsection{Null Controllability}
In this subsection, we  give a proof of Theorem \ref{thm:main}. Before giving the detailed proof of Theorem \ref{thm:main}, we introduce the main idea briefly.

We divide $(0,T)$ into $I_k, J_k, k\in\mathbb{N}^+$ such that 
\begin{equation*}
    (0,T)=\bigcup_{k\in\mathbb{N}^+}(I_k\cup J_k), \quad I_k=[T_k,\ \tilde T_k),\ J_k=[\tilde T_k, T_{k+1}),\ k\in\mathbb{N}^+, 
\end{equation*}
where $T_k, \tilde T_k\ (k\in\mathbb{N}^+)$ will be given precisely latter. 

At first, On each interval $I_k,k\in\N^+$, the coupled system with a control switching between $G_1\times E$ and $G_2\times F$ in an unknown mode is controlled. Then, On every interval $J_k,k\in\N^+$, we let the coupled system without control freely evolve.

On interval $I_1$, set the initial datum for coupled system with a control switching between $G_1\times E$ and $G_2\times F$ to be $(y_0,z_0)$. Then for each $k\in\N^+\setminus \{1\}$, the initial datum on $I_k$ is defined to be the ending value of the solution to the coupled system without control on $J_{k-1}$. The initial datum of the coupled system without control on $J_k,k\in\N^+$, is endowed by the ending value of the solution for the coupled system with a control switching from $G_1\times E$ to $G_2\times F$ on $I_k$. To this end, we need to know the ending values of the solution on every $I_k,k\in\N^+$.

On one hand, according to Proposition \ref{pro-partial}, for each $k\in\N^+$, we can search a control $u^{(k)}(\cdot)\in L^2(I_k;L^2(G_1\cup G_2))$  such that the corresponding solution $(y^{(k)}(\cdot),z^{(k)}(\cdot))$ to the equation on $I_k$ satisfies
$$
\Pi_{k}(y^{(k)}(\tilde{T}_k))=\bar{\Pi}_{k}(z^{(k)}(\tilde{T}_k))=0.
$$
Moreover, the estimate of the control $u^{(k)}(\cdot)$ is obtained.

On the other hand, by virtues of the decay estimate of the freely evolved coupled system, we can obtain a suitable norm estimate for the ending value of the solution to the equation on $J_k,k\in\N^+$.

Finally, from these, the control
$$
u(t)=
\begin{cases}
u^{(k)}(t),\quad &t\in I_k,\\
0, &t\in J_k,
\end{cases}
$$
drives the solution of coupled system (\ref{eq:main}) to rest at time $T$.

Now, we start with proving Theorem \ref{thm:main}.

\begin{proof}[Proof of Theorem \ref{thm:main}] We barrow some idea from \cite{wang2008null} and divide the whole proof into following parts.

{\it Step 1.} Let $k\in\N^+$. Set  $E=\cup_{i=1}^\infty E_i$, $E_i=(t_{2i-1},t_{2i})$, $F=\cup_{i=1}^\infty F_i$, $F_i=(t_{2i},t_{2i+1})$, such that $(0,T)=F\cup E$ and $E\cap F=\emptyset$. Denote
$$
T_k= 
\begin{cases}
0,  \ &k=1,\\[3mm]
T\sum_{i=1}^{k-1}2^{-i},  &k>1,
\end{cases}
\qquad
\tilde{T}_k=
\begin{cases}
\dfrac{T}{4},  &k=1,\\[3mm]
T\bigg(\sum_{i=1}^{k-1}2^{-i}+2^{-k-1}\bigg), &k>1.
\end{cases}
$$
From these, one can easily check that
\begin{equation}
\label{Tk}
\tilde{T}_k-T_k=T_{k+1}-\tilde{T}_k=T2^{-k-1} ~\text{for all}~  k\in\N^+.
\end{equation}
Define the following sequences of time intervals:
$$
I_k=[T_k,\tilde{T}_k),\ J_k=[\tilde{T}_k,T_{k+1}) ~\text{for all}~  k\in\N^+.
$$
Let $\rho_k^2=C_0^k$, where $C_0>32$ is a sufficiently large positive constant. For all $k\in\N^+$, write
\begin{equation*}
    \begin{split}
        \alpha_k
        &=\dfrac{C(\bar\lambda_k^2+\lambda_k^2)e^{\tau (\tilde{T}_k-T_k)}(e^{C\bar{\lambda}_{\rho_k}^\sigma}+e^{C\sqrt{\lambda_{\rho_k}}})}{\tilde{T}_k-T_k},\\
\beta_k
&=(\alpha_k+1)e^{\tau (\tilde{T}_k-T_k)},\\
\theta_k
&=e^{-\left(2\min\{ \lambda_{\rho_{k+1}},\bar{\lambda}_{\rho_{k+1}} \}-\tau\right)(T_{k+1}-\tilde{T}_k)}.
    \end{split}
\end{equation*}
Consider the following coupled system in $I_1=[T_1,\tilde{T}_1)$:
\begin{equation}
\label{th-eq1}
\begin{cases}
(y^{(1)})_t = A y^{(1)} + ay^{(1)}+ bz^{(1)}+ \chi_E\chi_{G_1} u^{(1)}, & \left( x,t\right) \in I\times I_1,   \\
(z^{(1)})_t = \bar{A} z^{(1)} + cy^{(1)}+ dz^{(1)}+ \chi_F\chi_{G_2} u^{(1)}, & \left( x,t\right) \in I\times I_1,   \\
y^{(1)}(T_1) =y_{0},z^{(1)}(T_1) =z_{0}, &  x\in I,\\
(0<\alpha<1)\
\begin{cases}
y^{(1)}(1,t)=z^{(1)}(1,t) = 0, & t\in \left(0,T\right), \\
y^{(1)}(0,t)=z^{(1)}(0,t) = 0, & t\in \left(0,T\right),
\end{cases}\\
(1\leq \alpha<2)\
\begin{cases}
y^{(1)}(1,t)=z^{(1)}(1,t) = 0, & t\in \left(0,T\right), \\
(x^\alpha (z^{(1)})_x)(0,t)=y^{(1)}(0,t) = 0, & t\in \left(0,T\right),
\end{cases}
\end{cases}
\end{equation}
and the following coupled system without control in $J_k=[\tilde{T}_k, T_{k+1})$, $\forall\,k\in\N^+$:
\begin{equation}
\label{th-eq2}
\begin{cases}
(\hat{y}^{(k)})_t = A \hat{y}^{(k)} + a\hat{y}^{(k)}+ b\hat{z}^{(k)}, & \left( x,t\right) \in I\times J_k,   \\
(\hat{z}^{(k)})_t = \bar{A} \hat{z}^{(k)} + c\hat{y}^{(k)}+ d\hat{z}^{(k)}, & \left( x,t\right) \in I\times J_k,   \\
\hat{y}^{(k)}(\tilde{T}_k) =y^{(k)}(\tilde{T}_k),\hat{z}^{(k)}(\tilde{T}_k) =z^{(k)}(\tilde{T}_k), &  x\in I,\\
(0<\alpha<1)\
\begin{cases}
\hat{y}^{(k)}(1,t)=\hat{z}^{(k)}(1,t) = 0, & t\in \left(0,T\right), \\
\hat{y}^{(k)}(0,t)=\hat{z}^{(k)}(0,t) = 0, & t\in \left(0,T\right),
\end{cases}\\
(1\leq \alpha<2)\
\begin{cases}
\hat{y}^{(k)}(1,t)=\hat{z}^{(k)}(1,t) = 0, & t\in \left(0,T\right), \\
(x^\alpha (\hat{z}^{(k)})_x)(0,t)=\hat{y}^{(k)}(0,t) = 0, & t\in \left(0,T\right),
\end{cases}
\end{cases}
\end{equation}
and the following coupled system in $I_k=[T_k,\tilde{T}_k)$, $\forall\,k\in\N^+\setminus\{1\}$:
\begin{equation}
\label{th-eq3}
\begin{cases}
(y^{(k)})_t = A y^{(k)} + ay^{(k)}+ bz^{(k)}+ \chi_E\chi_{G_1} u^{(k)}, & \left( x,t\right) \in I\times I_k,   \\
(z^{(k)})_t = \bar{A} z^{(k)} + cy^{(k)}+ dz^{(k)}+ \chi_F\chi_{G_2} u^{(k)}, & \left( x,t\right) \in I\times I_k,   \\
y^{(k)}(T_k) =\hat{y}^{(k-1)}(T_k),z^{(k)}(T_k) =\hat{z}^{(k-1)}(T_k), &  x\in I,\\
(0<\alpha<1)\
\begin{cases}
y^{(k)}(1,t)=z^{(k)}(1,t) = 0, & t\in \left(0,T\right), \\
y^{(k)}(0,t)=z^{(k)}(0,t) = 0, & t\in \left(0,T\right),
\end{cases}\\
(1\leq \alpha<2)\
\begin{cases}
y^{(k)}(1,t)=z^{(k)}(1,t) = 0, & t\in \left(0,T\right), \\
(x^\alpha (z^{(k)})_x)(0,t)=y^{(k)}(0,t) = 0, & t\in \left(0,T\right).
\end{cases}
\end{cases}
\end{equation}
We are going to prove that for each $k\in\N^+$, there exists a control $u^{(k)}\in L^2(I_k;L^2(G_1\cup G_2))$ such that
\begin{equation}
\label{th-eq4}
\Pi_{\rho_k}(y^{(k)}(\tilde{T}_k))=\bar{\Pi}_{\rho_k}(z^{(k)}(\tilde{T}_k))=0~\text{in}~ I
\end{equation}
with
\begin{equation}
\label{th-eq5}
\begin{cases}
\|y^{(1)}(\tilde{T}_1)\|^2+\|z^{(1)}(\tilde{T}_1)\|^2\leq\beta_1(\|y_0\|^2+\|z_0\|^2),\\
\|y^{(k)}(\tilde{T}_k)\|^2+\|z^{(k)}(\tilde{T}_k)\|^2\leq \beta_k\Pi_{i=1}^{k-1}\beta_i\theta_i(\|y_0\|^2+\|z_0\|^2),\,\,k\in\N^+\setminus\{1\},
\end{cases}
\end{equation}
and
\begin{equation}
\label{th-eq6}
\begin{cases}
\|u^{(1)}\|^2_{L^2(I_1;L^2(G_1\cup G_2))}\leq \alpha_1(\|y_0\|^2+\|z_0\|^2),\\
\|u^{(k)}\|^2_{L^2(I_k;L^2(G_1\cup G_2))}\leq \alpha_k\Pi_{i=1}^{k-1}\beta_i\theta_i(\|y_0\|^2+\|z_0\|^2),\,\,k\in\N^+\setminus\{1\}.
\end{cases}
\end{equation}

{\it Step 2.} We first do that for $k = 1$. By Proposition \ref{pro-partial}, for $\rho_1\in\N^+$, there exists a control $u^{(1)}\in L^2(I_1;L^2(G_1\cup G_2))$ , such that the corresponding solution $(y^{(1)}, z^{(1)})$ of (\ref{th-eq1}) satisfies
\begin{equation}
\label{th-eq7}
\Pi_{\rho_1}(y^{(1)}(\tilde{T}_1))=\bar{\Pi}_{\rho_1}(z^{(1)}(\tilde{T}_1))=0~\text{in}~ I,
\end{equation}
and
\begin{equation}
\label{th-eq8}
\|y^{(1)}(\tilde{T}_1)\|^2+\|z^{(1)}(\tilde{T}_1)\|^2\leq \beta_1(\|y_0\|^2+\|z_0\|^2),
\end{equation}
and
\begin{equation}
\label{th-eq9}
\|u^{(1)}\|^2_{L^2(I_1;L^2(G_1\cup G_2))}\leq \alpha_1(\|y_0\|^2+\|z_0\|^2).
\end{equation}
Therefore, for $k=1$, (\ref{th-eq4}),  the first conclusions of (\ref{th-eq5}) and (\ref{th-eq6}) are true.

{\it Step 3.} We should be to prove that (\ref{th-eq4}) and the second conclusions of (\ref{th-eq5}) and (\ref{th-eq6}) hold for $k = n + 1$, provided that (\ref{th-eq4}) and the second conclusions of (\ref{th-eq5}) and (\ref{th-eq6}) hold for $k = n$. However, in order the give a more readable proof, here we also prove that (\ref{th-eq4}) and the second conclusions of (\ref{th-eq5}) and (\ref{th-eq6}) are true for $k = 2$.

Let $(\hat{y}^{(1)},\hat{z}^{(1)})$ be the solution to the the following coupled system without control in $J_1=[\tilde{T}_1, T_2)$:
$$
\begin{cases}
(\hat{y}^{(1)})_t = A \hat{y}^{(1)} + a\hat{y}^{(1)}+ b\hat{z}^{(1)}, & \left( x,t\right) \in I\times J_1,   \\
(\hat{z}^{(1)})_t = \bar{A} \hat{z}^{(1)} + c\hat{y}^{(1)}+ d\hat{z}^{(1)}, & \left( x,t\right) \in I\times J_1,   \\
\hat{y}^{(1)}(\tilde{T}_1) =y^{(1)}(\tilde{T}_1),\hat{z}^{(1)}(\tilde{T}_1) =z^{(1)}(\tilde{T}_1), &  x\in I,\\
(0<\alpha<1)\
\begin{cases}
\hat{y}^{(1)}(1,t)=\hat{z}^{(1)}(1,t) = 0, & t\in \left(0,T\right), \\
\hat{y}^{(1)}(0,t)=\hat{z}^{(1)}(0,t) = 0, & t\in \left(0,T\right),
\end{cases}\\
(1\leq \alpha<2)\
\begin{cases}
\hat{y}^{(1)}(1,t)=\hat{z}^{(1)}(1,t) = 0, & t\in \left(0,T\right), \\
(x^\alpha (\hat{z}^{(1)})_x)(0,t)=\hat{y}^{(1)}(0,t) = 0, & t\in \left(0,T\right),
\end{cases}
\end{cases}
$$
From (\ref{th-eq7}), (\ref{th-eq8}) and Proposition \ref{pro-decay}, we see
\begin{equation}
\label{th-eq10}
\|\hat{y}^{(1)}(T_2)\|^2+\|\hat{z}^{(1)}(T_2)\|^2\leq \theta_1\big(\|y^{(1)}(\tilde{T}_1)\|^2+\|z^{(1)}(\tilde{T}_1)\|^2  \big)\leq \beta_1\theta_1\big( \|y_0\|^2+\|z_0\|^2 \big).
\end{equation}
Let $(y^{(2)},z^{(2)})$ be the solution to the the following coupled system in $I_2=[T_2,\tilde{T}_2)$:
$$
\begin{cases}
(y^{(2)})_t = A y^{(2)} + ay^{(2)}+ bz^{(2)}+ \chi_E\chi_{G_1} u^{(2)}, & \left( x,t\right) \in I\times I_2,   \\
(z^{(2)})_t = \bar{A} z^{(2)} + cy^{(2)}+ dz^{(2)}+ \chi_F\chi_{G_2} u^{(2)}, & \left( x,t\right) \in I\times I_2,   \\
y^{(2)}(T_2) =\hat{y}^{(1)}(T_2),z^{(2)}(T_2) =\hat{z}^{(1)}(T_2), &  x\in I,\\
(0<\alpha<1)\
\begin{cases}
y^{(2)}(1,t)=z^{(2)}(1,t) = 0, & t\in \left(0,T\right), \\
y^{(2)}(0,t)=z^{(2)}(0,t) = 0, & t\in \left(0,T\right),
\end{cases}\\
(1\leq \alpha<2)\
\begin{cases}
y^{(2)}(1,t)=z^{(2)}(1,t) = 0, & t\in \left(0,T\right), \\
(x^\alpha (z^{(2)})_x)(0,t)=y^{(2)}(0,t) = 0, & t\in \left(0,T\right),
\end{cases}
\end{cases}
$$
By Proposition \ref{pro-partial} and (\ref{th-eq10}), for $\rho_2\in\N^+$, there exists a control $u^{(2)}\in L^2(I_2;L^2(G_1\cup G_2))$, such that the solution $(y^{(2)}, z^{(2)})$ satisfies
$$
\Pi_{\rho_2}(y^{(2)}(\tilde{T}_2))=\bar{\Pi}_{\rho_2}(z^{(2)}(\tilde{T}_2))=0~\text{in}~ I,
$$
and
$$
\|y^{(2)}(\tilde{T}_2)\|^2+\|z^{(2)}(\tilde{T}_2)\|^2\leq \beta_2(\|\hat{y}^{(1)}(T_2)\|^2+\|\hat{z}^{(1)}(T_2)\|^2)\leq\beta_2\beta_1\theta_1\big( \|y_0\|^2+\|z_0\|^2 \big),
$$
and
$$
\|u^{(2)}\|^2_{L^2(I_2;L^2(G_1\cup G_2))}\leq \alpha_2(\|\hat{y}^{(1)}(T_2)\|^2+\|\hat{z}^{(1)}(T_2)\|^2)\leq\alpha_2\beta_1\theta_1\big( \|y_0\|^2+\|z_0\|^2 \big).
$$
Therefore, for $k=2$, (\ref{th-eq4}),  the second conclusions of (\ref{th-eq5}) and (\ref{th-eq6}) are true.

{\it Step 4.} We next prove that (\ref{th-eq4}) and the second conclusions of (\ref{th-eq5}) and (\ref{th-eq6}) are true for $k = n + 1$, on the condition that they are true for $k = n$. Here is the argument: Since there are $(y^{(n)},z^{(n)})$ and $u^{(n)}$
satisfy (\ref{th-eq3}), (\ref{th-eq4}) and the second conclusions of (\ref{th-eq5}) and (\ref{th-eq6}) for $k = n$, then equation (\ref{th-eq2}) for $k = n$, has an unique solution $(\hat{y}^{(n)},\hat{z}^{(n)})$, which alone with Proposition \ref{pro-decay} satisfying that
\begin{equation}
\label{th-eq11} 
\begin{split}
\|\hat{y}^{(n)}(T_{n+1})\|^2+\|\hat{z}^{(n)}(T_{n+1})\|^2
&\leq \theta_n\big(\|y^{(n)}(\tilde{T}_n)\|^2+\|z^{(n)}(\tilde{T}_n)\|^2  \big)\\
&\leq \Pi_{i=1}^n\beta_i\theta_i\big( \|y_0\|^2+\|z_0\|^2 \big).
\end{split}
\end{equation}
By Proposition \ref{pro-partial} and (\ref{th-eq11}), for $\rho_{n+1}\in\N^+$, there exists a control $u^{(n+1)}\in L^2(I_{n+1};L^2(G_1\cup G_2))$, such that the solution $(y^{(n+1)}, z^{(n+1)})$ satisfies
$$
\Pi_{\rho_{n+1}}(y^{(n+1)}(\tilde{T}_{n+1}))=\bar{\Pi}_{\rho_{n+1}}(z^{(n+1)}(\tilde{T}_{n+1}))=0~\text{in}~ I,
$$
and
\begin{equation*}
    \begin{split}
        \|y^{(n+1)}(\tilde{T}_{n+1})\|^2+\|z^{(n+1)}(\tilde{T}_{n+1})\|^2
        &\leq \beta_{n+1}(\|\hat{y}^{(n)}(T_{n+1})\|^2+\|\hat{z}^{(n)}(T_{n+1})\|^2)\\
&\leq \beta_{n+1}\Pi_{i=1}^n\beta_i\theta_i\big( \|y_0\|^2+\|z_0\|^2 \big), 
    \end{split}
\end{equation*}
and
\begin{equation*}
    \begin{split}
       \|u^{(n+1)}\|^2_{L^2(I_{n+1};L^2(G_1\cup G_2))}
       &\leq \alpha_{n+1}(\|\hat{y}^{(n)}(T_{n+1})\|^2+\|\hat{z}^{(n)}(T_{n+1})\|^2)\\
&\leq \alpha_{n+1}\Pi_{i=1}^n\beta_i\theta_i\big( \|y_0\|^2+\|z_0\|^2 \big).  
    \end{split}
\end{equation*}
Therefore, for $k=n+1$, (\ref{th-eq4}),  the second conclusions of (\ref{th-eq5}) and (\ref{th-eq6}) are true. From these, we know that (\ref{th-eq4}), (\ref{th-eq5}) and (\ref{th-eq6}) hold for all $k\in \N^+$.

{\it Step 5.} We show that there exists a constant $L > 0$ such that for all $k\in\N^+$,
\begin{equation}
\label{th-eq12}
\|u^{(k)}\|^2_{L^2(I_k;L^2(G_1\cup G_2))}\leq L \big( \|y_0\|^2+\|z_0\|^2 \big).
\end{equation}
We first estimate $\beta_k\theta_k$. By (\ref{Tk}), $\sigma>0$ and alone with the definition of $\rho_k$, we have
\begin{equation*}
    \begin{split}
        \beta_k\theta_k
        &=(C(\bar\lambda_{\rho_k}^2+\lambda_{\rho_k}^2)e^{\tau T2^{-k}2^{-1}}(e^{C\bar{\lambda}_{\rho_k}^\sigma}+e^{C\sqrt{\lambda_{\rho_k}}})T^{-1}2^k2+1)e^{\tau T2^{-k}2^{-1}}\\
&\hspace{3.5mm}\cdot e^{-(2\min\{ \lambda_{\rho_{k+1}},\bar{\lambda}_{\rho_{k+1}} \}-\tau)e^{\tau T2^{-k}2^{-1}}}  \\
&\leq Ce^{C}(e^{C(\rho_k)^{2\sigma}}+e^{C\rho_k})\cdot e^{-C2^{-k}\min\{(\rho_{k+1})^{2\sigma}, (\rho_k+1)^2\}}
\\
&\leq Ce^{C(\rho_k)^{2\sigma}}\cdot e^{-C2^{-k}(\rho_{k+1})^2}\leq Ce^{-C(2^{-k}\rho_{k+1}^2-\rho_{k+1}^{2\sigma})} \leq Ce^{-C2^k}. 
    \end{split}
\end{equation*}
Thus, there exists a $N_1 > 0$ such that for all $k\geq N_1$, it holds that $\beta_k\theta_k\leq1$. Next, We estimate $\alpha_{k+1}\beta_k\theta_k$.
\begin{equation*}
\begin{split}
\alpha_{k+1}\beta_k\theta_k
&=
C(\bar\lambda_{\rho_{k+1}}^2+\lambda_{\rho_{k+1}}^2)e^{\tau T2^{-k-1}2^{-1}}(e^{C\bar{\lambda}_{\rho_{k+1}}^\sigma}+e^{C\sqrt{\lambda_{\rho_{k+1}}}})T^{-1}2^{k+1}2\\
&\hspace{3.5mm}\cdot (C(\bar\lambda_{\rho_k}^2+\lambda_{\rho_k}^2)e^{\tau T2^{-k}2^{-1}}(e^{C\bar{\lambda}_{\rho_k}^\sigma}+e^{C\sqrt{\lambda_{\rho_k}}})T^{-1}2^k2+2)e^{\tau T2^{-k}2^{-1}}\\
&\hspace{3.5mm}\cdot e^{-(2\min\{ \lambda_{\rho_{k+1}},\bar{\lambda}_{\rho_{k+1}}\}-\tau)e^{\tau T2^{-k}2^{-1}}} \\
&\leq Ce^{C}(e^{C(\rho_{k+1})^{2\sigma}}+e^{C\rho_{k+1}})\cdot
Ce^{-C(2^{-k}\rho_{k+1}^2-\rho_{k+1}^{2\sigma})}\\
&\leq Ce^{C\rho_{k+1}^{2\sigma}}\cdot
Ce^{-C(2^{-k}\rho_{k+1}^2-\rho_{k+1}^{2\sigma})} \leq Ce^{-C(2^{-k}\rho_{k+1}^2-2\rho_{k+1}^{2\sigma})}\leq Ce^{-C2^k}.
\end{split}
\end{equation*}
Hence, there exists a $N_2 > 0$ such that for all $k\geq N_2$, it holds that $\alpha_{k+1}\beta_k\theta_k\leq1$. Therefore, we know that
$$
\alpha_{k+1}\Pi_{i=1}^k(\beta_i\theta_i)\leq1,\quad\ \forall\, k\geq\max\{ N_1,N_2 \}.
$$
Let
$$
L=2\max_{ k\leq \max\{ N_1,N_2,1 \}}\{\alpha_1,\alpha_{k+1}\Pi_{i=1}^k(\beta_i\theta_i),1 \}.
$$
Then we have that for all $k\in\N^+$, (\ref{th-eq12}) is true.

{\it Step 6.} We are going to construct the control which drives the solution of coupled system (\ref{eq:main}) to $0$ at time $t = T$. To achieve such a goal, we let
\begin{equation}
\label{th-eq14}
u(t)=
\begin{cases}
u^{(k)}(t), &t\in I_k,\\
0, &t\in J_k.
\end{cases}
\end{equation}
Then we have
$$
\|u\|^2_{L^2(0,T;L^2(G_1\cup G_2))}\leq  L\big( \|y_0\|^2+\|z_0\|^2 \big).
$$

Next, we shall prove that
\begin{equation}
\label{th-eq16}
y(T)=z(T)=0.
\end{equation}
Since there are $(y^{(k)},z^{(k)})$, $u^{(k)}$ and $(\hat{y}^{(k)},\hat{z}^{(k)})$ satisfy (\ref{th-eq2}), (\ref{th-eq3}), we can make
use of (\ref{th-eq14}) to obtain that
$$
y(\tilde{T}_k)=y^{(k)}(\tilde{T}_k),\ z(\tilde{T}_k)=z^{(k)}(\tilde{T}_k),~\text{for all}~ k\in\N^+,
$$
which together with (\ref{th-eq2}), yields that
\begin{equation}
\label{th-eq15}
\Pi_{\rho_k}(y(\tilde{T}_k))=\bar{\Pi}_{\rho_k}(z(\tilde{T}_k))=0~\text{in}~ I,~\text{for all}~ k\in\N^+.
\end{equation}
We arbitrarily fix a $n \in \N^+$. Then it follows from (\ref{th-eq15}) that
$$
\Pi_{\rho_n}(y(\tilde{T}_k))=\bar{\Pi}_{\rho_n}(z(\tilde{T}_k))=0~\text{in}~ I,~\text{for all}~ k\geq n.
$$
This implies that
\begin{equation}
\label{th-eq17} 
0=\lim_{k\rightarrow+\infty}\Pi_{\rho_n}(y(\tilde{T}_k))=\Pi_{\rho_n}(y(T)),\quad 
0=\lim_{k\rightarrow+\infty}\bar{\Pi}_{\rho_n}(z(\tilde{T}_k))=\bar{\Pi}_{\rho_n}(z(T)). 
\end{equation}
By the definition of $\rho_k$, we can pass to the limit for $n\rightarrow\infty$ in (\ref{th-eq16}) to get (\ref{th-eq17}). The proof of Theorem \ref{thm:main} is completed.
\end{proof}
\subsection{Negative Null Controllability}
Next, we show the lack of null controllability for the coupled system (\ref{eq:main}), presented in Theorem \ref{thm:negative-main}.

\begin{proof}[Proof of Theorem \ref{thm:negative-main}] We divide the proof into two Steps.

{\it Step 1.} Denote  $E=\cup_{i=1}^\infty E_i$, $E_i=(t_{2i-1},t_{2i})$, $F=\cup_{i=1}^\infty F_i$, $F_i=(t_{2i},t_{2i+1})$, such that $(0,T)=F\cup E$ and $E\cap F=\emptyset$. Set $\chi_E\equiv1$ and $c(\cdot)=0$ in $(0,T)$, a.e. Without loss of generality, we may assume that the coefficient $d(\cdot)$ in the coupled system (\ref{eq:main}) is equal to $0$ (Otherwise, we introduce a simple transformation $\tilde{y}=y$, $\tilde{z}(t)=e^{-\int_0^td(s)ds}z(t)$ and $\tilde{u}=u$ and consider the system for the new state variable $(\tilde{y},\tilde{z})$ and the control variable $\tilde{u}$). Then, by the coupled system (\ref{eq:main}), and noting that $c(\cdot)=d(\cdot)=0$ in $(0,T)$ a.e., we find that $(y,z)$ solves
\begin{equation}
\label{thm:negative-main-eq1}
\begin{cases}
y_t = A y + ay+ bz+ \chi_E\chi_{G_1} u, & \left( x,t\right) \in I\times(0,T),   \\
z_t = \bar{A} z , & \left( x,t\right) \in I\times(0,T),   \\
y\left(x, 0\right) =y_{0},z\left(x, 0\right) =z_{0}, &  x\in I,\\
(0<\alpha<1)\
\begin{cases}
y(1,t)=z(1,t) = 0, & t\in \left(0,T\right), \\
y(0,t)=z(0,t) = 0, & t\in \left(0,T\right),
\end{cases}\\
(1\leq \alpha<2)\
\begin{cases}
y(1,t)=z(1,t) = 0, & t\in \left(0,T\right), \\
(x^\alpha z_x)(0,t)=y(0,t) = 0, & t\in \left(0,T\right),
\end{cases}
\end{cases}
\end{equation}
Since there is no control in the second equation of coupled system (\ref{thm:negative-main-eq1}), $z$ cannot be driven to the rest for any time $T$ if $z_0 \neq 0$ in $I$. Thus, conclusion $(1)$ is true.

{\it Step 2.} The proof of conclusion $(2)$ is similar to Step 1. Therefore, the proof of Theorem \ref{thm:negative-main} is complete.

\end{proof}
\subsection{Observability Estimate}
In the following, we prove the observability estimate (\ref{eq:obs-inq}) for the adjoint coupled system (\ref{eq:adjoint}), i.e., Theorem \ref{thm:obs-inq}.

\begin{proof}[Proof of Theorem \ref{thm:obs-inq}] Set  $E=\cup_{i=1}^\infty E_i$, $E_i=(t_{2i-1},t_{2i})$, $F=\cup_{i=1}^\infty F_i$, $F_i=(t_{2i},t_{2i+1})$, such that $(0,T)=F\cup E$ and $E\cap F=\emptyset$. For any given terminal value $(p_T, w_T )\in \big(L^2(I)\big)^2$, denote by $(p,w)$ the corresponding solution of (\ref{eq:adjoint}). By (\ref{eq:main}) and (\ref{eq:adjoint}), we have
$$
(yw)_t=yw_t+wy_t,\quad\ (zp)_t=zp_t+pz_t.
$$
Integrating the above equalities over $(0,1) \times(0,T)$, respectively, it holds 
\begin{equation}
\label{thm-inq-1}
\begin{split}
&\int_0^T\int_{G_1}\chi_E uwdxdt+\int_0^T\int_{G_2}\chi_F updxdt\\
&=\int_0^1y(x,T)w_T(x)dx+\int_0^1z(x,T)p_T(x)dx-\int_0^1z_0(x)p(x,0)dx-\int_0^1y_0(x)w(x,0)dx.
\end{split}
\end{equation}
Let $y_0=-w(0)$, $z_0=-p(0)$ and $u$ be the control driving the solution $(y,z)$ to $0$ at time $t = T$ and such that
$$
\|u\|^2_{L^2(0,T;L^2(G_1\cup G_2)} \le L  (\|y_0\|^2+\|z_0\|^2).
$$
From (\ref{thm-inq-1}), it is easy to check that
\begin{equation}
\label{thm-inq-2} 
\int_0^T\int_{G_1}\chi_E uwdxdt+\int_0^T\int_{G_2}\chi_F updxdt=\int_0^1p^2(x,0)dx+\int_0^1w^2(x,0)dx.
\end{equation}

On the other hand, it follows
\begin{equation*}
\begin{split}
&\int_0^T\int_{G_1}\chi_E uwdxdt+\int_0^T\int_{G_2}\chi_F updxdt\\
&\leq\|\chi_E w\|_{L^2(0,T;L^2(G_1))}\|u\|_{L^2(0,T;L^2(G_1))}+\|\chi_Fp\|_{L^2(0,T;L^2( G_2))}\|u\|_{L^2(0,T;L^2(G_2))}\\
&\leq\big[\|\chi_E w\|_{L^2(0,T;L^2(G_1))}+\|\chi_Fp\|_{L^2(0,T;L^2( G_2))}\big]\|u\|_{L^2(0,T;L^2(G_1\cup G_2))}\\
&\leq C\big[\|\chi_E w\|_{L^2(0,T;L^2(G_1))}+\|\chi_Fp\|_{L^2(0,T;L^2( G_2))}\big]\big( \|p(0)\|^2+\|w(0)\|^2 \big)^{\frac{1}{2}},
\end{split}
\end{equation*}
which, alone with (\ref{thm-inq-2}), yields
\begin{equation*}
\int_0^1p^2(x,0)dx+\int_0^1w^2(x,0)dx\leq C\bigg[\int_0^T\int_{G_1}\chi_Ew^2dxdt +\int_0^T\int_{G_2}\chi_F p^2dxdt\bigg].
\end{equation*}
The proof is complete.
\end{proof}

\section{Acknowledgments}
This work is supported by the National Natural Science Foundation of China (Grants No.
11871478), the Science Technology Foundation of Hunan Province.

\section{Funding}
This work is supported by the National Natural Science Foundation of China (Grants No.
11871478), the Science Technology Foundation of Hunan Province.

\section{Declarations}
The authors have not disclosed any competing interests.

\bibliographystyle{abbrvnat}
\bibliography{ref.bib}

\end{document}